\documentclass[a4paper, english, 10pt]{amsart}

\usepackage[OT2,OT1]{fontenc} \newcommand\cyr
{
\renewcommand\rmdefault{wncyr} \renewcommand\sfdefault{wncyss} \renewcommand\encodingdefault{OT2} \normalfont
\selectfont
}
\DeclareTextFontCommand{\textcyr}{\cyr}    

\usepackage{amsfonts}
\usepackage{subfig} 
\usepackage{caption}
\usepackage{forest}

\usepackage{amsmath}
\usepackage{amssymb}
\usepackage{tikz}
\usepackage[cp850]{inputenc}
\usepackage[bookmarksnumbered,plainpages
]{hyperref}
\usepackage{color}
\usepackage{mathtools}
\usepackage{MnSymbol}
\usepackage[all]{xy}

\allowdisplaybreaks
\newtheorem{theorem}{Theorem}[section]

\newtheorem{Proposition}[theorem]{Proposition}

\newtheorem{proposition}[theorem]{Proposition}
\newtheorem{corollary}[theorem]{Corollary}

\newtheorem{remark}[theorem]{Remark}

\newtheorem{lemma}[theorem]{Lemma}

\newtheorem{definition}[theorem]{Definition}
\newtheorem{example}[theorem]{Example}

\def\sym#1{\mathrm{Sym}(#1)}
\def\c#1{\mathrm{con}_{#1}}

\def\aut#1{\mathrm{Aut}(#1)}

\def\aff#1{\mathrm{Aff}#1}
\def\Aff#1{\mathrm{Aff}#1}
\def\lmlt{\mathrm{LMlt}}
\def\dis{\mathrm{Dis}}
\def\Q{\mathcal{Q}_{\mathrm{Hom}}}
\def\GL#1#2{\mathrm{GL}_{#1}(#2)}
\def\setof#1#2{\{#1\, : \,#2\}}

\def\P{\mathcal{P}}
\newcommand{\BlocS}[2]{\dis(#1)_{[#2]}}
\newcommand{\LP}[1]{\langle L_{#1}\rangle^\P}

\newcommand*\xbar[1]{%
   \hbox{%
     \vbox{%
       \hrule height 0.5pt 
       \kern0.5ex
       \hbox{%
         \kern-0.1em
         \ensuremath{#1}%
         \kern-0.1em
       }%
     }%
   }%
} 

\makeindex 
\subjclass[2010]{17D99, 20N02.}

\keywords{Quandles, automorphisms of quandles, left-distributivity.}

\title{Principal and Doubly Homogeneous Quandles}
\author{Marco Bonatto}

\address[M. Bonatto]{IMAS--CONICET and Universidad de Buenos Aires, 
Pabell\'on~1, Ciudad Universitaria, 1428, Buenos Aires, Argentina}

\email{marco.bonatto.87@gmail.com}

\begin{document}
\maketitle

\section*{Abstract}
In the paper we describe the class of principal quandles and we show that connected quandles can be decomposed as a disjoint union of principal quandles. We also prove that simple affine quandles are finite and they can be characterized among finite simple quandles by several different equivalent properties, such as for instance being doubly homogeneous (i.e. having a doubly transitive automorphism group). A complete description of finite doubly-homogeneous quandles is provided extending the result of \cite{V} and solving \cite[Problem 6.7]{2trans}. We also provide a classification of connected cyclic quandles with several fixed points independently from \cite{CQ}.
\section*{Introduction}

Quandles are binary algebras developed in connection with knot theory \cite{J,Matveev}, Hopf Algebras \cite{AG} and the set theoretical solution of the Yang-Baxter equation \cite{EGS,ESS}. Most of the construction and tools used in quandle theory are based on groups \cite{EGS,HSV} and modules \cite{Hou}. In a recent paper \cite{CP} we developed a commutator theory for racks and quandles in the sense of \cite{comm}, where the notion of {\it abelianness} and {\it nilpotence} are developed for general algebras. In this paper we are going to use this universal algebraic viewpoint to investigate some classes of quandles: {\it principal}, {\it doubly homogeneous} and {\it cyclic}  quandles (see Section \ref{Sec: principal}, \ref{Sec: doubly} and \ref{Sec:cyclic} respectively). \\
The class of {\it principal} quandles naturally generalizes the class of {\it affine} or {\it Alexander} quandles and it has also been studied under different names as {\it generalized Alexander} quandles \cite{C2,GenAlex}. Moreover they are provided by a {\it right linear} construction over groups (in the sense of \cite{Stanos}).  The study of principal quandles provides useful tools to understand the other classes of quandles investigated in the present paper and quandles in general.

It is known that every quandle with at least four elements and a $3$-transitive automorphism group is projection \cite[Proposition 5]{McC} (recall that a group $G$ acting on a set $Q$ is called $k$-transitive if it is transitive on the set of $k$-tuples with different entries). Quandles with transitive automorphism groups are called {\it homogeneous} quandles and they are characterized by a well-known construction over groups \cite[Theorem 7.1]{J}. The characterization of quandles with $k$-transitive automorphism group for $k\geq 2$ is an open problem stated in \cite[Problem 6.7]{2trans}. The case $k\geq 3$ was studied and solved in \cite{Russi2}. We call quandles with $2$-transitive automorphism group {\it doubly homogeneous} and their charaterization in still open. A partial answer to this problem is given in \cite{V}, where the class of finite quandles with doubly transitive left multiplication group (i.e. doubly transitive quandles) is completely understood (see \cite[Corollary 4]{V}). These quandles are also characterized by a precise cyclic structure of the left multiplications given by one fixed point and one non-trivial cycle. This class is contained into the class of {\it doubly homogeneous} quandles and into the class of {\it cyclic quandles}, defined and investigated in \cite{CQ} as the class of finite quandles which cycle structure of left multiplications is given by several fixed points and one non trivial cycle. In this paper we provide a characterization of the first class (answering to  \cite[Problem 6.7]{2trans}) and we show an alternative proof of the classification of connected cyclic quandles (the same classification is obtained in \cite{CQ} using combinatorial techniques).

The contents of the paper are organized as follows: in Section \ref{Sec: prelim} we collect some basic results on quandles and we introduce two new relations for quandles: the first one is related to projection subquandles and the second one was already defined in \cite{CP} in connection with central congruences. We define the class of {\it semiregular quandles} as the class of quandles with semiregular displacement group (of which the class of principal quandles is an instance). We extend some known result for principal quandles to this class (indeed they depend just on the semiregularity of the displacement group, see Section \ref{semiregular quandles}).

In Section \ref{Sec: principal} we study congruences and automorphisms of connected principal quandles and we show that every connected quandle is a disjoint union of isomorphic copies of a principal quandle (Theorem \ref{decomposition}). In Section \ref{Sec: isog} we deal with right linear left distributive quasigroups, studied under the name of {\it isogroups} \cite{GALK2,Vla}. The finite algebras of this class are the finite principal latin quandles.\\
The class of {\it affine} quandles is one of the most studied class of quandles and it is an example of a class of principal quandles: we show  that there are no infinite simple affine quandles (Theorem \ref{no infinite simple}). 

In Section \ref{Sec: doubly} we investigate a particular class of simple quandles, namely {\it strictly simple} quandles (i.e. quandles with no proper subquandle). Theorem \ref{minimal iff} shows that finite simple quandles are latin if and only if they are strictly simple, establishing the converse of Theorem 3.4 of \cite{Va} for quandles, extending a result known for left-distributive quasigroup \cite{GALK}. Moreover finite simple latin quandles are affine and doubly homogeneous and their representation over elementary abelian groups is explained in Proposition \ref{min are affine} (already shown in \cite{AG, JS} as the family of simple quandles with abelian displacement group). Moreover, strictly simple quandles are the only finite simple principal quandles (indeed a principal simple quandle $Q$ is faithful and then latin by virtue of Proposition \ref{On maximal proj sub}). 

Using the principal decomposition for quandles we show that a finite quandle is doubly homogeneous if and only if it is a power of a strictly simple quandle (see Theorem. \ref{doubly-hom are power of min}. In the proof we make use of Maschke theorem for representations of finite groups).

%
%
In Section \ref{sec_extensions} we investigate extensions of strictly simple quandles by projection quandles of prime size (in the same direction of \cite{Claw}) and we provide strong constrains on the size of such extensions in Theorem \ref{extension of minimal by proj of prime}. In Section \ref{Sec:cyclic} we show an independent proof of the classification of finite connected cyclic quandles using the contents of the previous Section (Theorem \ref{cyclic}).
\subsection*{Notation and terminology}
Given an algebra $A$ (i.e. a set with an arbitrary set of operations)  we denote by $\aut{A}$ the group of its automorphism. A congruence of $A$ is an equivalence relation which respects the algebraic structure. The congruences of $A$ form a lattice denoted by $Con(A)$ with minimum $0_Q=\setof{(a,a)}{a\in Q}$ and maximum $1_Q=A\times A$. Note that homomorphic images and congruences are essentially the same thing since kernels of homomorphisms are congruence. If $\alpha\in Con(A)$ the algebra $A/\alpha$ is called factor algebra. We denote by $\setof{\gamma_n(A)}{n\in \mathbb{N}}$ and by $\setof{\gamma^{n}(A)}{n\in \mathbb{N}}$ the congruences in the lower central series and of the derived series of $A$ and by $\zeta_A$ its center (defined in analogy with group theory to capture the notion of solvability and centrality, see \cite{CP, comm}).\\
By $\textbf{H}(A)$, $\textbf{S}(A)$ and $\textbf{P}(A)$ we denote respectively the set of isomorphism classes of factors, of subalgebras and of powers of $A$. 
For further details about the universal algebraic definitions and construction we refer the reader to \cite{UA}. 

We denote group operations just by juxtaposition or by $+$ if the operation is commutative. Given an element $g$ of a group $G$ we denote by $\widehat{g}$ its inner automorphism, by $N_G(S)$ (and resp. $C_G(S)$) the normalizer (resp. centralizer) of a subset $S\subseteq G$. The core of a subgroup $H$ is the biggest normal subgroup of $G$ contained in $H$ and it is denoted by $Core_G(H)$. If $\rho$ is an action of $G$ on a set $Q$ we denote the point-wise stabilizer of $a$ in $G$ by $G_a$ and the set-wise point-stabilizer of a subset $S$ by $G_S=\setof{g\in G}{g(S)=S}$. The orbit of $a$ under the action of $G$ will be denoted by $a^G$. A group is called semiregular if $G_a=1$ for every $a\in G$ and regular if it semiregular and transitive. Note that the pointwise-stabilizer of a transitive group is normal if and only if it is trivial (the stabilizers of a transitive group are all conjugate).
\section{Preliminary results}\label{Sec: prelim}
\subsection{Quandles}
A quandle is an idempotent left-distributive left quasigroup, i.e. it is a binary algebra $(Q,\ast,\backslash)$ such that
\begin{eqnarray*}
a\ast (a\backslash b) &=& a\backslash (a\ast b)=b,\\
a\ast( b\ast c)&=&(a\ast b)\ast (a\ast c),\\
a\ast a &=& a.
\end{eqnarray*}
for every $a,b\in Q$. Equivalently a quandle is binary algebra $(Q,*)$ such that the left multiplication mapping $L_a:b\mapsto a\ast b$ is in the stabilizer of $a$ of the automorphism group of $Q$ for every $a\in Q$.
\begin{example}\label{examples}
\begin{itemize}
\item[(i)] Any set $Q$ with the operation $a\ast b=b$ for every $a,b\in Q$ is a \emph{projection quandle}. If $|Q|=n$ we denote such quandle by $\mathcal{P}_n$.
\item[(ii)] Let $G$ be a group and $H\subseteq G$ a subset closed under conjugation. For $a$, $b\in H$, let $a\ast b = aba^{-1}$. Then $Conj{(H)}=(H,\ast,\backslash)$ is a quandle, called \emph{conjugation quandle} over $H$.
\end{itemize}
\end{example}
\begin{example}\label{example2}
Let $G$ be a group, $f\in  \aut{G}$ and $H \leq Fix(f)=\setof{a\in G}{f(a)=a}$. Let $G/H$ be the set of left cosets of $H$ and the multiplication defined by
\begin{displaymath}
   aH\ast bH=af(a^{-1}b)H.
\end{displaymath}
The algebra $(G/H,*)$ is a quandle  denoted by $\Q(G,H,f)$. If $H=1$ then $Q$ is called \emph{principal} (over $G$) and it will be denoted just by $\Q(G,f)$. If $G$ is abelian then $Q$ is called \emph{affine} (over $G$) and in this case we use the notation $\aff(G,f)$. 
\end{example}
 The group generated by the left multiplications mappings is called the {\it left multiplication group} and it is denoted by $\lmlt(Q)$. The {\it displacement group}, defined as $\dis(Q)=\langle \setof{L_a L_b^{-1}}{a,b\in Q}\rangle$ is a very important object in the theory of quandles as many properties of quandles are determined by its group-theoretical properties. Moreover $\lmlt(Q)=\dis(Q)\langle L_a\rangle$ for every $a\in Q$ and so the orbits of $\dis(Q)$ and $\lmlt(Q)$ coincide and they are called connected components. \\
 A quandle $Q$ is said to be (doubly) {\it homogeneous} if $\aut{Q}$ acts (doubly) transitively and {\it connected} if $\dis(Q)$ acts transitively. Clearly every connected quandle is homogeneous, but the orbits with respect these two groups can be very different. Let $\mathcal{P}_n$ be the projection quandle of size $n$, then $\aut{\mathcal{P}_n}=\sym{n}$ and $\dis(\mathcal{P}_n)=1$. Hence, it is homogeneous but totally disconnected.

The construction given in Example \ref{example2} completely characterizes homogeneous quandles \cite[Theorem 7.1]{J}, and so this construction is also called homogeneous representation. 
%
%
Connected quandles have a representation over their displacement, i.e. any connected quandle $Q$ is isomorphic to $\Q(\dis(Q),\dis(Q)_a,\widehat{L}_a)$ \cite[Theorem 4.1]{HSV}. Moreover, connected abelian quandles are affine \cite[Theorem 2.2]{JZ2}.

Congruences of a quandle $Q$ induce surjective morphisms between the displacement group of $Q$ and the displacement group of its factors. Indeed the mapping
\begin{displaymath}
\pi_\alpha :\dis(Q)\longrightarrow \dis(Q/\alpha),\quad L_{x} L_{y}^{-1} \mapsto L_{[x]}\ldots L_{[y]}^{-1}
\end{displaymath}
can be extended to a well-defined surjective group morphism for every $\alpha\in Con(Q)$ (see \cite[Lemma 1.8]{AG}). The kernel of $\pi_\alpha$ has the following characterization:
\begin{displaymath}
\dis^{\alpha}=\setof{h\in \dis(Q)}{[h(a)] =[a], \text{ for every } a\in Q}= \bigcap_{a\in Q/\alpha} \dis(Q)_{[a]},
\end{displaymath}
where $\dis(Q)_{[a]_\alpha}$ is the set-wise stabilizer in $\dis(Q)$. Note that the stabilizer of a block of $\alpha$ is $\dis(Q)_{[a]_\alpha}=\pi_\alpha^{-1}(\dis(Q/\alpha)_{[a]_\alpha})$ and if $Q/\alpha$ is connected then $\dis^\alpha=Core_{\dis(Q)}(\dis(Q)_{[a]})$ for every $a\in Q$. The set-wise stabilizer can be used to check connectedness of a quandle (we will use this criterion with no further reference through all the paper):
\begin{lemma}\label{block stab} \cite[Proposition 1.3]{GB}
Let $Q$ be a quandle and $\alpha\in Con(Q)$. Then	$Q$ is connected if and only if $Q/\alpha$ is connected and $\dis(Q)_{[a]_\alpha}$ is transitive on $[a]_\alpha$ for every $a\in Q$.
\end{lemma}
For every congruence $\alpha$ we can define the {\it displacement group relative to the congruence $\alpha$} as $\dis_\alpha=\langle\setof{L_a L_b^{-1}}{a\, \alpha\, b}\rangle$ which is a normal subgroup of $\lmlt(Q)$ contained in the kernel $\dis^\alpha$. Properties as abeliannes and centrality of congruences are completely determined by the properties of the relative displacement groups (see \cite[Theorem 1.1]{CP}).

On the other hand for every subgroup in $Norm(Q)=\setof{N\triangleleft \lmlt(Q)}{N\leq \dis(Q)}$ we can define the congruence $\c{N}=\setof{(a,b)\in Q\times Q}{L_a L_b^{-1}\in N}$. The pair of mappings $\alpha\mapsto \dis_\alpha$, $N\mapsto \c{N}$ provides a Galois correspondence between $Con(Q)$ and $Norm(Q)$ and its properties are described in Section 3.3 of \cite{CP}. This correspondence can be used to obtain information on the displacement group from the congruence lattice and vice versa. Moreover the orbit decomposition with respect to the action of $N\in Norm(Q)$ is also a congruence and we denote it by $\mathcal{O}_N$ \cite[Lemma 2.6]{CP}.

The set $L(Q)=Conj(\setof{L_a}{a\in Q})$ is a conjugation quandle and $\Lambda_Q:Q\to L(Q)$ is a surjective quandle morphism. We denote by $\lambda_Q$ the kernel of such homomorphism. A quandle is called:
\begin{itemize}
\item[(i)]{\it faithful } if $\Lambda_Q$ is injective, i.e. $\lambda_Q=0_Q$;
\item[(ii)] {\it crossed set} if $a\ast b=b$ implies $b\ast a=a$ for every $a,b\in Q$;
\item[(iii)] {\it latin} if all right multiplications $R_a: b\mapsto b\ast a$ are bijective.
\end{itemize}
Note that any latin quandle is faithful and any faithful quandle is a crossed set. If $Q$ is faihtful the pointwise stabilizer of an element $a\in Q$ is $\dis(Q)_a=Fix(\widehat{L}_a)$. \\ 
Latin quandles can be also understood as left distributive (LD) quasigroup including right division $a/b=R_b^{-1}(a)$ as a basic operation, but congruence and subalgebras of $(Q,*,\backslash)$ and $(Q,*,\backslash,/)$ might be different. Nevertheless if $Q$ is finite nothing changes.\\
We refer the reader to \cite{AG,HSV,J} for further details on the basics of quandle theory. 
\subsection{Projection subquandles}
 Let $Q$ be a quandle and let $\P$ be the relation defined as
\begin{equation}\label{relation P}
a \, \P \, b \quad \text{ if and only if } \{a,b\}=\mathcal{P}_2.
\end{equation}
 The relation $\P$ is symmetric and reflexive, but in general is not an equivalence. Let $[a]_\P=\setof{b\in Q}{a\, \P\, b}$. 
 Clearly $[a]_{\lambda_Q}\subseteq [a]_{\P}\subseteq Fix(L_a)$ and $[a]_\P$ is a subquandle of $Q$. Indeed
if $b,c\in [a]_\P$, then $b,c\in Fix(L_a)$ and $L_a$ and $L_b$ commute. Therefore 
\begin{eqnarray*}
L_b^{\pm 1}(c)\ast a &=& L_b^{\pm 1} L_c L_b^{\mp 1}(a)=L_b^{\pm 1} L_c L_b^{\mp 1}(a)=a\\ 
a\ast L_b^{\pm 1}(c) &=& L_a L_{b}^{\pm 1}(c)=L_b^{\pm 1}L_a(c)=L_b^{\pm 1}(c).
\end{eqnarray*}
Hence, $L_b^{\pm 1}(c)\in [a]_\P$. A quandle $Q$ is a crossed set if and only if $[a]_{\P}=Fix(L_a)$ for every $a\in Q$.

Projection subquandles of $Q$ containing $a\in Q$ are contained in $[a]_\P$. Therefore, a quandle has no projection subquandles if and only if $\P=0_Q$. Moreover, the action of the automorphism group respects the relation $\P$. Indeed, if $\{a,b\}$ is a projection subquandle, so it is $\{h(a),h(b)\}$ for every $h\in \aut{Q}$. Hence we can define:
\begin{eqnarray}
\mathrm{Aut}^{\P}&=&\setof{h\in\aut{Q}}{h(a)\, \P\, a\, \text{for every }a\in Q}, \notag\\
 \lmlt^\P &=& \lmlt(Q)\bigcap \mathrm{Aut}^\P,\notag
\end{eqnarray}
Note that both $\mathrm{Aut}^\P$ and $\lmlt^\P$ are normal subgroups of $\aut{Q}$ and the orbit decomposition with respect to the action of $\lmlt^\P$ is a congruence of $Q$ (see \cite[Theorem 6.1]{CP}). 

The relation $\mathcal{P}$ is related to the cycle structure of left multiplications.
\begin{proposition}\label{proj sub of crossed}
Let $Q$ be a quandle. The following are equivalent:
\begin{itemize}
\item[(i)] $Fix(L_a)=\{a\}$ for every $a\in Q$.
\item[(ii)] $\P=0_Q$ i.e. $Q$ has no projection subquandle.
\item[(iii)] All the subquandles of $Q$ are faithful.
\end{itemize}
\end{proposition}
\begin{proof} (i) $\Rightarrow$ (ii) If $\{a,b\}$ is projection, then $b\in Fix(L_a)=\{a\}$, i.e. $a=b$ 

(ii) $\Rightarrow$ (iii) Let $M$ be a subquandle of $Q$. The blocks of $\lambda_M$ are projection subquandles of $Q$, so they are trivial and $M$ is faithful.

(iii) $\Rightarrow$ (i) In particular $Q$ is faithful and so it is a crossed set. If $b\in Fix(L_a)$, i.e. $a\ast b=b$, then also $b\ast a=a$, so $M=\{a,b\}$ is a faithful projection subquandle of $Q$. Therefore $a=b$.
\end{proof}

\subsection{Semiregular quandles}\label{semiregular quandles}
One of the most studied classes of quandles is the class of connected affine quandles. Some of their particular properties depend just on the fact that the displacement group is semiregular. Thus, we study the family of quandles with such property. A quandle is called {\it semiregular} if $\dis(Q)$ is semiregular on $Q$, namely $\dis(Q)_a=1$ for every $a\in Q$. This family is relevant since every quandle decomposes as a disjoint union of semiregular quandles (see Proposition \ref{semiregular decomposition}).

Semiregularity of a quandle $Q$ is captured by the equivalence relation already defined in \cite{CP} as follows:
\begin{equation}\label{def_sigma}
a \, \sigma_Q\, b \quad \text{ if and only if } \quad  \dis(Q)_a=\dis(Q)_b.
\end{equation}
Indeed $Q$ is semiregular if and only if $\sigma_Q=1_Q$. It is easy to see that the class of semiregular quandles is closed under $\textbf{P}$ and $\textbf{S}$.

\begin{proposition}\label{semiregular decomposition}
Let $Q$ be a  quandle. The classes of $\sigma_Q$ are semiregular subquandles of $Q$ and they are blocks with respect to the action of $\aut{Q}$. In particular every quandle is a disjoint union of semiregular quandles.
\end{proposition}
\begin{proof}
Let $a\, \sigma_Q \, b$, i.e. $ \dis(Q)_a=\dis(Q)_b$ and $h\in \aut{Q}$. Then $\dis(Q)_{h(a)}=h \dis(Q)_a h^{-1}=h \dis(Q)_b h^{-1} = \dis(Q)_{h(b)}$. Hence $h(a)\, \sigma_Q\, h(b)$.\\
Moreover $\dis(Q)_{L_a^{\pm 1}(b)}=L_a^{\pm 1} \dis(Q)_b L_a=L_a^{\pm 1} \dis(Q)_a L_a=\dis(Q)_a$, so $[a]_{\sigma_Q}$ is a subquandle of $Q$. The action of the displacement group of $[a]_{\sigma_Q}$ is the action of the group $D=\langle L_b L_a^{-1}, \,b\in [a]_{\sigma_Q}\rangle $ restricted to $[a]_{\sigma_Q}$. So if $h=g|_{[a]_{\sigma_Q}}\in \dis([a]_{\sigma_Q})_a$ for 
some $g\in D$, then $g(b)=h(b)=b$ for every $b\in [a]_{\sigma_Q}$. Therefore $h=1$ and $[a]_{\sigma_Q}$ is semiregular. 
\end{proof}
Let $Q$ be a quandle, $a\in Q$ and let $N_a=N_{\dis(Q)}(\dis(Q)_a)$. Then $a^{\dis(Q)}\bigcap [a]_{\sigma_Q}=a^{N_a}$. Indeed $h(a)\, \sigma_Q\, a$ if and only if $\dis(Q)_{h(a)}=h \dis(Q)_a h^{-1} =\dis(Q)_a$, i.e. $h\in N_a$. Moreover $\zeta_Q=\sigma_Q\bigcap \c{Z(\dis(Q))}\leq \sigma_Q$, according to \cite[Proposition 5.9]{CP}.
The equivalence $\sigma_Q$ can be trivial and in this case $N_a=\dis(Q)_a$ and $Z(\dis(Q))=1$ ($Z(\dis(Q))\leq N_a$ for every $a\in Q$).

\begin{Proposition}\label{semiregular factor} Let $Q$ be a quandle and $\alpha\in Con(Q)$.
\begin{itemize}
\item[(i)] $Q/\alpha$ is semiregular if and only if $\dis^\alpha=\dis(Q)_{[a]_\alpha}$ for every $a\in Q$.
\item[(ii)]  Let $\beta= \bigwedge_{i\in I} \alpha_i\in Con(Q)$. If $Q/\alpha_i$ is semiregular for every $i\in I$, then $Q/\beta$ is semiregular.
\end{itemize}
\end{Proposition}

\begin{proof}
(i) Since $\dis(Q)_{[a]_\alpha}=\pi_\alpha^{-1}(\dis(Q/\alpha)_{[a]})$, then $\dis(Q/\alpha)$ is semiregular if and only if $\dis^\alpha=\dis(Q)_{[a]_\alpha}$ for every $a\in Q$. 

(ii) Let $\beta=\bigwedge_{i\in I}\alpha_i$. Note that $\dis(Q)_{[a]_\beta}=\bigcap_{i\in I}\dis(Q)_{[a]_{\alpha_i}}$ and according to \cite[Proposition 3.2(ii)]{CP}, we have $\dis^{\beta}=\bigcap_{i\in I} \dis^{\alpha_i}$. Using item (i) we get
 $$\dis(Q)_{[a]_\beta}= \bigcap_{i\in I}\dis(Q/\alpha_i)_{[a]_{\alpha_i}}=\bigcap_{i\in I}\dis^{\alpha_i} =\dis^\beta.$$
Therefore we can conclude that $Q/\beta$ is semiregular by (i).
\end{proof}
Note that if $Q/\alpha$ is semiregular then $\dis(Q)_a\leq \dis(Q)_{[a]}=\dis^\alpha$ for every $a\in Q$.

\begin{lemma}\label{semidirect}
Let $Q$ be a semiregular quandle. Then $\lmlt(Q)\cong \dis(Q)\rtimes \langle L_a\rangle$ for every $a\in Q$.
\end{lemma}
\begin{proof}
The displacement group of $Q$ is semiregular, therefore $\dis(Q)\bigcap \langle L_a\rangle\leq\dis(Q)_a=1$ and so $\lmlt(Q)=\dis(Q)\rtimes \langle L_a\rangle$ since $\lmlt(Q)=\dis(Q)\langle L_a\rangle$ and $\dis(Q)$ is normal.
\end{proof}

The following lemma shows that for semiregular quandles the equivalence $\P$ is a congruence. Note that the implication (iii) $\Rightarrow$ (ii) $\Rightarrow$ (i) holds in general
\begin{proposition}\label{On maximal proj sub}
Let $Q$ be a semiregular quandle. Then $Q$ is a crossed set and $\P=\lambda_Q$. Therefore, the following are equivalent:
\begin{itemize}
\item[(i)] $Q$ is faithful.
\item[(ii)] $Q$ has no projection subquandles.
\item[(iii)] Right multiplications are injective.
\end{itemize}
\end{proposition}
\begin{proof}
 Let $b\in Fix(L_a)$. Then $L_b^{-1} L_a(b)=b$  i.e. $L_b^{-1} L_a\in \dis(Q)_b=1$. So $L_a=L_b$ i.e. $a \lambda_Q \,b$ and $a\in Fix(L_b)$. Therefore $Q$ is a crossed set and $\P=\lambda_Q$.
 
(i) $\Leftrightarrow$ (ii) Clear since $\P=\lambda_Q$.


(i) $\Rightarrow$ (iii) Let $a\ast b=c\ast b$. Then $L_c^{-1}L_a(b)=b$, so $L_c=L_a$ and so $a=c$. 

(iii) $\Rightarrow$ (i) If $L_a=L_b$, then $a\ast b=b\ast b$ and so $a=b$. 
\end{proof}

For a semiregular quandle $Q$ the blocks of $\lambda_Q$ are the maximal projection subquandles of $Q$. 
%

\begin{corollary}\label{principal}
Finite semiregular faithful quandles are latin.
\end{corollary}
\begin{proof}
If $Q$ is finite and faithful, then right multiplications are injective and so bijective, i.e. $Q$ is latin.
\end{proof}

\section{Principal quandles}\label{Sec: principal}

\subsection{Principal decomposition}
Principal quandles form a particular class of homogeneous quandles. Recall that a principal quandle $Q$ is defined as $Q=\Q(G,f)$ and the quandle operation is 
\begin{equation}\label{op for principal}
a\ast b=af(a^{-1}b),
\end{equation}
for every $a,b\in G$ (if $Q$ is affine then $a\ast b=(1-f)(a)+f(b)$). The binary algebras defined as in \eqref{op for principal} are also called {\it right linear} over the group $G$ (for further details see \cite{Stanos}). The canonical left action of $G$ on itself is a regular action by automorphisms of $Q$. 
The action of the generators of $\dis(Q)$ is given by
$$L_b L_1^{-1}(a)=bf(b)^{-1} a,$$
for every $a,b\in Q$. Then $\dis(Q)\simeq [G,f]=\langle \setof{af(a)^{-1}}{a\in G}\rangle\leq G$ and its action is given by the canonical left action on $G$ which is semiregular. Therefore principal quandles are semiregular. The connected component of $a$ is $[G,f]a$ for every $a\in G$. Principal quandles can be characterized as follows:
\begin{proposition}\label{principal and regular}\cite[Corollary B.3]{GenAlex}
Let $Q$ be a quandle and $a\in Q$. The following are equivalent:
\begin{enumerate}
\item[(i)] $G\leq \aut{Q}$ is a regular and $\widehat{L_a}$-invariant subgroup.
\item[(ii)] $Q\cong \Q(G,\widehat{L}_a)$.
\end{enumerate}
If $Q$ is connected, the following are equivalent:
\begin{enumerate}
\item[(i)] $\dis(Q)$ is regular.
\item[(ii)] $Q$ is principal.
\end{enumerate}
In particular if $Q\cong \Q(G,f)$, then $G=[G,f]\cong \dis(Q)$.
\end{proposition}
\begin{proof}
(i) $\Rightarrow$ (ii) The group $G$ is regular and provides a homogeneous representation given by $Q\cong  \Q(G,\widehat{L_a})$, since $G_a= 1 $.

(ii) $\Rightarrow$ (i) The left action of $G$ is a regular action by automorphisms of $Q$.

Assume that $Q$ is connected. If $\dis(Q)$ is regular, then $Q\cong \Q(\dis(Q),\widehat{L_a})$, i.e. $Q$ is principal. On the other hand, if $Q$ is connected and principal then $\dis(Q)$ is semiregular and transitive, then regular. \\
 Let $Q=\Q(G,f)$ be a principal representation of $G$. Then $1^{\dis(Q)}= [G,f]=G$ and therefore $G\cong \dis(Q)$.
\end{proof}

\begin{corollary}\label{Latin principal}
Let $Q$ be a latin quandle and $a\in Q$. The following are equivalent: 
\begin{enumerate}
\item[(i)] $Q$ is principal.
\item[(ii)] $\dis(Q)=\{L_b L_{a}^{-1}, \ b\in Q\}$.
\end{enumerate}\end{corollary}
\begin{proof}
For a latin quandle $Q$ a set of representatives of cosets with respect to $\dis(Q)_a$ is $\setof{L_b L_{a}^{-1}}{b\in Q}$. So a latin quandle $Q$ is principal if and only if (ii) holds, by virtue of Proposition \ref{principal and regular}.
\end{proof}

\begin{corollary}\label{product of principal}
The class of principal quandles is closed under direct product.
\end{corollary}

\begin{proof}
Let $\setof{Q_i=\Q(G_i,f_i)}{i\in I}$ be a set of principal quandles. The group $\prod_{i\in I} G_i$ is regular on $\prod_{i\in I} Q_i$ and it is invariant under $f=\widehat{L}_{\setof{a_i}{i\in I}}=\prod_{i\in I} \widehat{L_{a_i}}$. So we can apply Proposition \ref{principal and regular} and $\prod_{i\in I} Q_i\cong \Q(\prod_{i\in I} G_i, f)$.
\end{proof} 

Every connected quandle is a disjoint union of principal subquandles (the following Theorem is the counterpart of Proposition \ref{semiregular decomposition} for connected quandles).
\begin{theorem}\label{decomposition}
Let $Q$ be a connected quandle. Then $[a]_{\sigma_Q}=a^{N_a}$ and it is a principal quandle over $N_a/\dis(Q)_a$ for every $a\in Q$.  In particular $Q$ is the disjoint union of isomorphic copies of a principal quandle.
\end{theorem}
\begin{proof}
Since $Q$ is connected $[a]_{\sigma_Q}\bigcap a^{\dis(Q)}=[a]_{\sigma_Q}=a^{N_a}$. Moreover $N_a$ is transitive over $[a]_{\sigma_Q}$, it contains $\dis(Q)_a$ as a normal subgroup and $h|_{[a]_{\sigma_Q}}(a)=a$ if and only if $h|_{[a]_{\sigma_Q}}=1$. So $G={N_a}|_{[a]_{\sigma_Q}}\cong N_a/\dis(Q)_a$ is regular over $[a]_{\sigma_Q}$ and since it is stable under the inner automorphism of $f={L_a}|_{[a]_{\sigma_Q}}$ then it provides a principal representation as $[a]_{\sigma_Q}\cong \Q(G,\widehat{f}|_G)$.\\
The quandle $Q$ is the union of the classes with respect to $\sigma_Q$ which are principal subquandles. Since they are blocks with respect to $\dis(Q)$ and the action is transitive, they are all isomorphic.
\end{proof}

\begin{remark}
Principal quandles are semiregular. On the other hand the connected components of semiregular quandles are principal (a principal representation is given by the displacement group). So every quandles is actually a disjoint union of principal quandles. In general they are not guaranteed to be isomorphic.
\end{remark}

%
%
Proposition \ref{On maximal proj sub} and Corollary \ref{principal} apply to principal quandles, so in particular they are crossed set. Principal quandles are homogeneous so left and right multiplications are all conjugate. Indeed if $Q$ is a quandle then $L_{h(a)}=hL_a h^{-1}$ and $R_{h(a)}=h R_a h^{-1}$ for every $h\in \aut{Q}$ and $a\in Q$. Hence item (i) of Proposition \ref{proj sub of crossed} is equivalent to $Fix(f)=Fix(L_1)=1$ and item(iii) of Proposition \ref{On maximal proj sub} is equivalent to injectivity of the map $R_1:a\mapsto af(a)^{-1}$.
The following is a direct Corollary of Proposition \ref{On maximal proj sub} using that the orbits of displacement group in an affine quandle $Q=\aff(A,f)$ are the coset with respect to $Im(1-f)$ and that $R_0=1-f$ (it can be found in \cite{HSV}).
\begin{corollary}\label{affine}
Let $Q=\aff(A,f)$ be an affine quandle. The following are equivalent: 
\begin{itemize}
\item[(i)] $Q$ is latin. 
\item[(ii)] $Q$ is connected and faithful.
\end{itemize}
If $Q$ is finite and connected, then it is latin.
\end{corollary}

%
%
\subsection{Subquandles and Automorphisms of principal quandles}
Principal quandles are homogeneous, so every subquandle of $Q=\Q(G,f)$ is given by $aS$ where $a\in G$ and $S$ is a subquandle containing $1$. So, up to isomorphism it is enough to consider subquandles containing $1$.

\begin{lemma}\label{on sub of princ}
Let $Q=\Q(G,f)$ be a principal quandle and $\alpha\in Con(Q)$.
\begin{itemize}
\item[(i)] 
Connected subquandles of $Q$ are cosets of a $f$-invariant subgroup of $G$ and they are principal.
\item[(ii)] If $Q$ is connected then $[a]_\alpha$ is principal over $\dis(Q)_{[a]_\alpha}$.
\end{itemize}

\end{lemma}
\begin{proof}
(i) Let $M$ be a connected subquandle of $Q$, without loss of generality we can assume that $1\in M$. 
Then $M=1^{\dis(M)}=\dis(M)=\langle\setof{af(a)^{-1}}{a\in M}\rangle$, so $\dis(M)$ is a $f$-invariant subgroup since $f(a)=1*a\in M$ for every $a\in M$, it is regular on $M$ and then $M$ is principal.

(ii) If $Q$ is connected the group $N=\dis(Q)_{[a]_\alpha}$ is regular over $[a]_\alpha$ and it is stable under $f=\widehat{L}_a$. Therefore, $[a]_\alpha\cong \Q(N,f|_N)$ is principal.
\end{proof}

If $Q=\Q(G,f)$ is finite  the size of non projection subquandles of $Q$ and the size of $Q$ are not coprime (indeed a subquandle $M$ containing $1$ is union of cosets with respect to $\dis(M)$ which is an $f$-invariant subgroup of $G$) and the size of connected subquandles of $Q$ divide the size of $Q$.

%

According to Lemma \ref{semidirect} $\lmlt(\Q(G,f))\cong [G,f]\rtimes \langle f\rangle $. The following Proposition shows the structure of the automorphism group of connected principal quandles extending \cite[Corollary 1.25]{AG} and \cite[Proposition 2.1]{Hou_Aut}.
\begin{proposition}\label{Prop:Aut of principal}
Let $Q=\Q(G,f)$ be a connected quandle. Then $\aut{Q}\cong G\rtimes C_{\aut{G}}(f)$.
\end{proposition}

\begin{proof}
Let $h\in \aut{Q}$, such that $h(1)=b$ and let $\lambda_b$ the left action of $b$ on $G$. Since $Q$ is connected $\lambda_b\in \dis(Q)$ and then $\lambda_b^{-1}h\in \aut{Q}_1$. So $\aut{Q}= \dis(Q)\aut{Q}_1$ and since $\dis(Q)$ is normal and $\dis(Q)\bigcap \aut{Q}_1=1$, we have that $\aut{Q}\cong G\rtimes \aut{Q}_1$.

Let $h\in \aut{Q}_1$ and let $t(a)=R_1(a)=af(a)^{-1}$. Then $L_{h(1)}=h L_1 h^{-1}=L_1$ and $h(t(a))=h(a\ast 1)=h(a)\ast 1= R_1(h(a))=t(h(a))$. The quandle $Q$ is connected, so $G=[G,f]=\langle \setof{t(a)}{a\in G}\rangle$. Hence we have:
\begin{eqnarray*}
 h(t(a_1)^{k_1}\ldots t(a_n)^{k_n})&=& h((L_{a_1}L_1^{-1})^{k_1}\ldots (L_{a_n} L_1^{-1})^{k_n}(1))=\\
&=& (L_{h(a_1)}L_1^{-1})^{k_1}\ldots (L_{h(a_n)} L_1^{-1})^{k_n}(1)=\\
&=& t(h(a_1))^{k_1}\ldots t(h(a_n))^{k_n}=\\
&=& h(t(a_1))^{k_1}\ldots h(t(a_n))^{k_n}.
\end{eqnarray*}
Therefore $h\in C_{\aut{G}}(f)$. The inclusion $C_{\aut{G}}(f)\leq \aut{Q}_1$ is clear. 
\end{proof}
%
%
%
\subsection{Congruence lattice of connected principal quandles}

Connected semiregular quandles are principal, so we can apply Proposition \ref{semiregular factor} to understand principal factor of connected quandles.
\begin{Proposition}\label{principal factor} Let $Q$ be a connected quandle and $\alpha\in Con(Q)$.
\begin{itemize}
\item[(i)] $Q/\alpha$ is principal if and only if $\dis^\alpha=\dis(Q)_{[a]_\alpha}$ for every $a\in Q$.
\item[(ii)]  Let $\beta= \bigwedge_{i\in I} \alpha_i\in Con(Q)$. If $Q/\alpha_i$ is principal for every $i\in I$, then $Q/\beta$ is principal.
\end{itemize}
\end{Proposition}

%
%
The class of (connected) principal quandles is not closed under $\textbf{H}$. Indeed SmallQuandle(12,i) with $i=1,2$ in the RIG library of GAP are principal connected quandles with a non-principal factor of size $6$.

Let $Q=\Q(G,f)$. Let us denote by $\sim_N$ the partition of $G$ given by left cosets with respect to $N\leq G$ and define
\begin{equation}\label{f-invariant and cong}
Sub(G,f)=\setof{N\leq G}{f(N)=N \text{ and }[N,f]\leq Core_G(N)}.
\end{equation}
Note that every subgroup of $Fix(f)$ and every $f$-invariant normal subgroup is in $Sub(G,f)$. Moreover characteristic subgroups of $G$ are contained in $Sub(G,f)$ for every $f\in \aut{G}$. Subgroups in \eqref{f-invariant and cong} provides congruences for every principal quandles and if a principal quandle is connected all its congruences arise in this way. 
\begin{theorem}\label{congruence for principal}
Let $Q=\Q(G,f)$ be a quandle. Then $$\setof{\sim_N}{N\in Sub(G,f)}\subseteq Con(Q).$$
If $Q$ is connected then $Con(Q)=Sub(G,f)$.
\end{theorem}

\begin{proof}
Let $N\in Sub(G,f)$, $a,b\in Q$ and let $n,m\in N$. Then
	\begin{eqnarray*}
		an*bm&=&a\underbrace{nf(n)^{-1}}_{\in Core_G(N)}f(a)^{-1}f(b)f(m)=af(a)^{-1}f(b)n^\prime f(m)=\\
		&=& (a*b)n^\prime f(m)\in (a*b)N.
	\end{eqnarray*}
	Then $an*bm\, \sim_N \, a*b$. Using that $n f^{-1}(n)^{-1}=f^{-1}(nf(n)^{-1})^{-1}\in [N,f]$ we can prove similarly that $an \backslash bm\, \sim_N\, a\backslash b$. Then $\sim_N$ is a congruence of $Q$.
	
	 Assume that $Q$ is connected and let $\alpha\in Con(Q)$ and $N=\dis(Q)_{[1]}$. Then $[1]=N$ is an $f$-invariant subgroup of $G$ and $\dis(Q)_{[a]}=a\dis(Q)_{[1]}a^{-1}$. Therefore $b\,\alpha\, a$ if and only if $b=a g a^{-1} a=ag$ for some $g\in  N$, i.e. $[a]=a N$. Moreover 
\begin{equation}
[N,f]=\langle af(a)^{-1},\, a\in N\rangle =\langle L_a L_1^{-1},\, a\in [1]\rangle \leq \dis_\alpha\leq \dis^\alpha=Core_G(N).
\end{equation}
So every congruence of $Q$ arises as the left cosets partition with respect to a subgroup in $Sub(G,f)$.
\end{proof}

\begin{corollary}\label{normal sub and cong}
Let $Q=\Q(G,f)$ be a connected quandle. The following are equivalent:
\begin{itemize}
\item[(i)]  $Q/\sim_N$ is principal. 
\item[(ii)] $N\trianglelefteq G$. 
\end{itemize}
If the blocks of $\sim_N$ are connected then $Q/\sim_N$ is principal and $\dis_{\sim_N}=\dis^{\sim_N}=N$.
\end{corollary}
\begin{proof}
The subgroup $N=\dis(Q)_{[1]}$ contains $\dis^{\sim_N}$ as its core. Using Proposition \ref{principal factor} (i) we have that $Q/\sim_N$ is principal if and only if $N=\dis^{\sim_N}$, i.e. $N$ is normal.\\
If $[1]_{\sim_N}$ is connected, then $\dis_{\sim_N}$ is transitive on $[1]_{\sim_N}$ and so $1^{\dis_{\sim_N}}=\dis_{\sim_N}=1^N=N$ and then $N$ is normal. Accordingly $Q/\sim_N$ is principal.
\end{proof}

\begin{example}
Let $Q=\Q(G,f)$ be a connected quandle and let $N=\setof{a\in G}{af(a)^{-1}\in Z(G)}$. A straightforward computation shows that $\zeta_Q=\c{Z(\dis(Q))}=\sim_N$. 
\end{example}
\subsection{Isogroups}\label{Sec: isog}

A latin quandle $(Q,*,\backslash)$ can be interpreted also as left-distributive (LD) quasigroup $(Q,*,\backslash,/)$ by adding right division as a basic operation. Note that if $Q$ is finite case the two structures are actually ({\it term}) equivalent since if $n$ is the order of right multiplications of $Q$ then 
$$x/y=R_y^{-1}(x)=R_y^{n-1}(x)=(\ldots(x*\underbrace{y)*y)\ldots )*y}_{n-1}.$$
In this section we will consider principal latin quandles as LD quasigroups and we call them {\it isogroups} following \cite{GALK2,Vla}. Isogroups are the LD quasigroups which are right linear over a group or equivalently they are {\it isotopic} to a group (see \cite{Bel} and \cite{Bruck} for loop isotopy). An isogroup $Q$ will be denoted as $Q=\Q(G,f,/)$. In particular $Q$ is ({\it polynomially}) equivalent to the algebra $(G,\cdot,{}^{-1},1,f,f^{-1}, t, t^{-1})$ where $t$ is the mapping $t: a\mapsto af(a)^{-1}$ (all the basic operations of the first algebra can be defined by using the basic operations of the second and some constant elements as $x\cdot y = x/1*1\backslash y$. See \cite[Section 4.3]{UA} for formal definitions of term and polynomial equivalence). The following results easily follows by the polynomial equivalence and they are available in in Russian in \cite{GALK2} and in English in \cite{Vla}.

\begin{Proposition}\label{factor of principal latin} 
Let $Q=\Q(G,f,/)$ be an isogroup. Then 
$$Con(Q)=Sub(Q,f)=\setof{\sim_N}{N\unlhd G,\ f(N)=N}$$
and all its factors are isogroups.
\end{Proposition}


\begin{proposition}\label{subquandle with zero are subgroups}
Let $Q=\Q(G,f,/)$ be an isogroup. The following are equivalent:
\begin{enumerate}
\item[(i)] $S\subseteq Q$ is a subalgebra of $Q$.
\item[(ii)] $S$ is a coset with respect to a subgroup invariant under $f$.
\end{enumerate}
In particular, the subalgebras of $Q$ are isogroups.
\end{proposition}

%
An alternative proof of Proposition \ref{subquandle with zero are subgroups} is given by Lemma \ref{on sub of princ}(ii), since every subalgebra of $Q=\Q(G,f,/)$ is also a connected subquandle of the principal quandle reduct $\Q(G,f)$ (the algebra obtained by $Q$ forgetting the $/$ operation).

If $Q=\Q(G,f,/)$ is an isogroup and $1\in S\subseteq G$, the subquandle (resp. congruence) generated by $S$ is the smallest (resp. normal) $f$-invariant subgroup of $G$ containing $S$ i.e. $\langle f^j(s), \, s\in S, \, j \, \in\mathbb{Z}\rangle$ (resp. $\langle gf^j(s)g^{-1}, \, s\in S, \, g\in G,\, j \, \in\mathbb{Z}\rangle$). 
Note that if $Q$ is affine quandle, then the lattice of congruences and the lattice of subquandles containing $1$ are the same (since conjugation is trivial).\\
 All the previous results apply to finite principal latin quandles. In particular every finite connected quandle with no projection subquandles is the disjoint union of isomorphic copies of a finite principal latin quandles (combining Theorem \ref{decomposition} and Corollay \ref{principal}).
 
A finite algebra has the {\it Lagrange property} if the size of every subalgebra divides the size of the algebra and it has the {\it Sylow property} if for every prime $p$ such that $p^n$ is the maximal power of $p$ dividing its size there exists a subalgebra of size $p^n$.
\begin{proposition}\cite[Proposition 1.4.5]{Vla}
\label{Lagrange and Sylow}
Finite principal latin quandles have the Lagrange and the Sylow property.
\end{proposition}
These properties do not extend to (connected) principal quandles. Odd order non faithful quandles have projection subquandles of size $2$ and in the RIG library there exist principal quandles for which the Sylow property does not hold (SmallQuandle(36,i) with $i=9,\ldots,14$).

In \cite[Section 3.4]{CP} we investigate quandles for which the pair $(\dis,\c{})$ are mutually inverse isomorphisms of lattice, under the name quandles with CDSg property. A class of such quandles is the class of finite principal latin quandles.
\begin{Proposition}\label{latin CDSg} Finite principal latin quandles have the CDSg property.
\end{Proposition}
\begin{proof}
Let $Q$ be a latin quandle. All the factors of $Q$ are latin and principal, hence faithful. Moreover, by Corollary \ref{principal factor} we have $\dis_{\alpha}=\dis^\alpha$ for every congruence $\alpha$. Then we can apply \cite[Proposition 3.10]{CP}. 
\end{proof}

\begin{Proposition}
Let $Q$ be a finite nilpotent quandle with the CDSg property. Then $Q$ is a principal latin quandle.
\end{Proposition}
\begin{proof}
If a quandle $Q$ has the CDSg property then $Q$ is faithful, connected and $\dis_\alpha=\dis^\alpha$ for every $\alpha\in Con(Q)$ and this property is stable under taking homomorphic images. If $Q$ is abelian then it is an affine connected and faithful quandle, hence latin. Assume that $Q$ is nilpotent of length $n+1$, i.e. $\gamma_n(Q)\leq \zeta_Q$. By induction on the nilpotency length, $Q/\gamma_n(Q)$ is principal and latin. So $\dis(Q)_a\leq \BlocS{Q}{a}=\dis^{\gamma_n(Q)}=\dis_{\gamma_n(Q)}\leq Z(\dis(Q))$ (the relative displacement group of a central congruence is central \cite[Theorem 1.1]{CP}). Therefore $\dis(Q)_a$ is normal and then trivial, so $Q$ is principal. The quandle $Q$ is finite, principal and faithful and so by Corollary \ref{principal} it is latin.
\end{proof} 

%
\subsection{Simple Affine quandles}
An affine quandle $\aff(A,f)$ carries a natural structure of $\mathbb{Z}[t,t^{-1}]$ modules, where the action of $t$ is given by the automorphism $f$. According to Proposition \ref{congruence for principal}, congruences of affine connected quandles correspond to submodules, indeed all subgroups are normal.
\begin{Proposition}\label{lattice affine}
Let $Q=\aff(A,f)$ be a connected affine quandle. Then 
$$Con(Q)=\setof{\sim_N}{N\leq A, \, f(N)=N},$$
i.e. congruences of $Q$ correspond to $\mathbb{Z}[t,t^{-1}]$-submodules of $Q$.
\end{Proposition}
\begin{proof}
Every $f$-invariant subgroup $N$ satisfies $[N,f]=Im(1-f)|_{N}\leq Core_A(N)=N$. Then we apply Theorem \ref{congruence for principal}.
\end{proof} 

\begin{remark}\label{remark on cong of finite affine}
If $Q=\aff(A,f)$ is finite and connected, then it is latin and there is a one-to-one correspondence between subquandles of $Q$ containing $0$, congruences of $Q$ and $f$-invariant subgroups of $A$ (i.e. submodules). 
\end{remark}

In order to understand simple affine quandles we can think of them as simple modules according to Proposition \ref{lattice affine}. 
%
%
%

\begin{proposition}
Let $Q$ be a simple affine quandle. Then $\dis(Q)\cong \mathbb{K}^n$ where $\mathbb{K}\in \{\mathbb{Q},\mathbb{Z}_p\}$. 
\end{proposition}
\begin{proof}
The displacement group of $Q$ has no characteristic subgroup. Therefore, either $n\dis(Q)=0$ for some $n$ and so it has finite exponent or $n\dis(Q)=\dis(Q)$ for every $n\in\mathbb{N}$ and so it is divisible. In the first case the exponent is $p$ and so $Q$ is a simple $\mathbb{Z}_p[t,t^{-1}]$-module. In the second case $\dis(Q)$ is a direct product of a power of $\mathbb{Q}$ and of its torsion part. Since $\dis(Q)$ has no characteristic subgroup, then $\dis(Q)$ is torsion free and so it is a power of $\mathbb{Q}$ and hence it is a simple $\mathbb{Q}[t,t^{-1}]$ module. If $\mathbb{K}$ is a field the ring $\mathbb{K}[t,t^{-1}]$ is a principal ideal domain and its simple modules are finite dimensional $\mathbb{K}$-vector spaces (as $\mathbb{K}[t,t^{-1}]$ is a localization of the polynomial ring $\mathbb{K}[t]$). \end{proof}
\begin{theorem}\label{no infinite simple}
Let $Q$ be an affine quandle and $|Q|>2$. Then $Q$ is simple if and only if $Q\cong \aff(\mathbb{Z}_p^{n},f)$ and $f$ acts irreducibly.
\end{theorem}
\begin{proof}
Let $A=\mathbb{Q}^n$ and let $Q=\aff(A,f)$. Then $f$ is a $\mathbb{Q}$-linear map and let $F$ be the matrix with coefficients in $\mathbb{Q}$ with respect to a basis $e_1,\ldots,e_n$. Let $k$ be the l.c.m. of the denominators of $\setof{F_{i,j},F^{-1}_{i,j}}{1\leq i,j \leq n}$. Then $F$ and $F^{-1}$ are elements of the subring of the matrices with coefficients in $\mathbb{Z}[k^{-1}]$. The subgroup $K=\langle f^j(e_1),\, j\in \mathbb{Z}\rangle$ is a $\mathbb{Z}[t,t^{-1}]$-submodule of $A$ and it is contained in $\mathbb{Z}[k^{-1}]^n$ which is a proper subring of $\mathbb{Q}^n$. Therefore $Q$ is not simple. Hence necessarily $\dis(Q)\cong \mathbb{Z}_p^n$ and $Q\cong \Aff(\mathbb{Z}_p^n,f)$ is simple if and only if $f$ acts irreducibly.
\end{proof}

\section{Doubly Homogeneous quandles}\label{Sec: doubly}
%
\subsection{Strictly simple quandles}
In this section we investigate strictly simple quandles and we prove one of the main results in Theorem \ref{minimal iff}.  A proper subquandle of $Q$ is a subquandle for which the undelying set is a proper subset $Q$ with more than one element.
\begin{definition} A quandle is said to be {\it strictly simple} if it has no proper subquandles.
\end{definition}
Minimal proper subquandles (with respect to inclusion) of finite quandles 
are strictly simple quandles. Strictly simple quandles are simple, since the blocks of congruences are subquandles and so they are connected and faithful \cite[Lemma 1]{J}.
\begin{lemma}\label{min are 2 generated}
A quandle $Q$ is strictly simple if and only if it is generated by any pair of its elements. Moreover $\lmlt(Q)$ is a Frobenius group.
\end{lemma}
\begin{proof}
Let $a,b\in Q$, then $Sg(a,b)$ has more than one element and so it is equal to $Q$. If $h\in \lmlt(Q)_a\bigcap	\lmlt(Q)_b$, then $h=1$ and so $\lmlt(Q)$ is a Frobenius group (it is transitive and it contains $L_a\in\lmlt(Q)_a$ for every $a\in Q$).
\end{proof}
\begin{theorem}\label{min are affine}
Let $Q$ be a finite quandle and $|Q|> 2$. Then $Q$ is strictly simple if and only if $Q\cong \Aff(\mathbb{Z}_p^{n}, f)$ and $f$ acts irreducibly.
\end{theorem} 

\begin{proof}
$(\Rightarrow)$	The group $\lmlt(Q)$ is a Frobenius group (Lemma \ref{min are 2 generated}). By Theorem 1 of \cite{Frob2}, it has a normal regular nilpotent subgroup $N$ which provides a homogeneous representation for $Q$. The quandle $Q$ is connected, so $\dis(Q)$ embeds into a quotient of $N$ (\cite[Theorem 4.1]{HSV}). Therefore $\dis(Q)$ is nilpotent and $Q$ is a simple nilpotent quandle \cite[Theorem 1.2]{CP}. Simple nilpotent algebras are abelian, and so $Q$ is affine and then we can conclude by Theorem \ref{no infinite simple}.

$(\Leftarrow)$ Assume that $Q\cong \aff(\mathbb{Z}_p^{n}, f)$ and $f$ has no proper invariant subgroups. By Corollary \ref{subquandle with zero are subgroups}, $Q$ has no proper subquandles.
\end{proof}
According to Proposition \ref{min are affine}, strictly simple quandles correspond to irreducible representations of cyclic groups. Irreducible cyclic subgroups lie in irreducible cyclic subgroups of maximal order, called {\it Singer cycle}, and it is known that the order of a Singer cycle in $\GL{n}{p}$ is $p^n-1$. An example of a Singer cycle is the group of linear transformation given by $\setof{\lambda: a\mapsto \lambda a	}{\lambda\in \mathbb{F}_{q}^*}$ where $\mathbb{F}_q$ is the finite field on $p^n$ elements. According to the isomorphism theorem in \cite{Nel}, $\aff(A,f)\cong \aff(B,g)$ if and only if $f$ and $g$ are conjugate in $\aut{A}$. All Singer cycles are conjugate \cite[Theorem 2.3.3]{256}, so strictly simple quandles are isomorphic to $\aff(\mathbb{F}_q,\lambda)$ for some $\lambda\in \mathbb{F}_q^*$.\\ 
Recall that if $G$ be a group acting on a set $Q$, $G$ is doubly transitive if and only if $G_a$ is transitive on $Q\setminus\{ a\} $ for every $a\in Q$. 
%
%
%
%
\begin{proposition}\label{min are 2trans}
Let $Q$ be a strictly simple quandle of size $p^n$. Then $\aut{Q}\cong \mathbb{Z}_p^n\rtimes \mathbb{Z}_{p^n-1}$ and $Q$ is doubly-homogeneous. \end{proposition}
\begin{proof}
Let $Q=\Aff(\mathbb{Z}_p	^n,f)$ be a strictly simple quandle. Since $f$ acts irreducibly on $\mathbb{Z}_p^n$ then $C=C_{GL_p(n)}(f)$ is a Singer cycle \cite[Theorem 2.3.5]{256} and then $C\cong \mathbb{Z}_{p^n-1}$. In particular, according to Corollary \ref{Prop:Aut of principal}, $\aut{Q}_0=C$, so it has size $p^n-1$. The quandle $Q$ is generated by $0$ and every $a\in Q$, so $\aut{Q}_{0,a}=1$. Therefore $|a^C|=|C|=p^n-1$ i.e. $\aut{Q}_0$ is transitive over $Q\setminus \{0\}$.
\end{proof}
The connected components of conjugation quandle over the automorphism group of a strictly simple quandle are affine connected quandles. 
\begin{proposition}\label{component of aut}
Let $Q=\aff(\mathbb{Z}_p^n,f)$ be a strictly simple quandle and let $C=Conj(\aut{Q})$. Then the connected component of $(a,g)$ in $C$ is isomorphic to the connected quandle $\aff(\mathbb{Z}_p^n,g)$ for every $g \neq 1$.
\end{proposition}
\begin{proof}
The automorphism group of $Q$ is $\aut{Q}\cong \mathbb{Z}_p^n\rtimes \mathbb{Z}_{p^n-1}$. Therefore
\begin{displaymath}
(b,h)(a,g)(b,h)^{-1}=(b+h(a),hg)(h^{-1}(-b),h^{-1})=((1-g)(b)+h(a),g).
\end{displaymath}
If $g\neq 1$ then $1-g$ is a bijection, since $\aut{Q}_0\bigcap \aut{Q}_b=1$ for every $b\in Q$. Then the orbit of $(a,g)$ is $\mathbb{Z}_p^n\times \{g\}$, therefore $(a,g)^{\lmlt(C)}\cong \Aff(\mathbb{Z}_p^n,g)$ which is connected since $Fix(g)=0$.
\end{proof}
In particular, if $Q$ is a strictly simple quandle of size $p^n$, then $Conj(\aut{Q})$ contains every strictly simple quandle of size $p^n$ (indeed they are isomorphic to $\aff(\mathbb{Z}_p^n,g)$ where $g\in C_{GL_n(p)}(f)$ acts irreducibly).
\subsection{Finite simple latin quandles}
In this section we give a characterization of finite simple latin quandles in terms of several equivalent properties.
First we show that finite doubly-homogeneous quandles are either projection or latin.
\begin{lemma}\label{Lem: 2trans are Latin}
Let $Q$ be a finite doubly homogeneous quandle. Then $Q$ is either projection or latin.
\end{lemma}
\begin{proof}
Let $Q$ be not a projection quandle, then for every $a\in Q$ there exists $b\in Q$ such that $b\ast a\neq a$. Let $c\in Q$, then there exists $f \in \aut{Q}_a$ such that $c=f(b\ast a)$, since $\aut{Q}$ is doubly transitive. Then $c=f(b\ast a)=f(b)\ast a=R_a(f(b))$. So $R_a$ is surjective and then bijective for every $a\in Q$.
\end{proof}
%
%
%
Using Corollary \ref{min are 2trans} and Lemma \ref{Lem: 2trans are Latin} we can finally prove one of the main theorems. 
\begin{theorem}\label{minimal iff}
Let $Q$ be a finite simple quandle and $|Q|>2$. The following are equivalent:
\begin{enumerate}
\item[(i)] $Q$ is abelian.
\item[(ii)] $Q$ is strictly simple.
\item[(iii)] $Q$ is doubly-homogeneous. 
\item[(iv)] $Q$ is latin.
\end{enumerate}
\end{theorem}

\begin{proof}
(i) $\Rightarrow$ (ii) By Theorem 3.4 of \cite{Va}, every finite simple abelian algebra has no proper subalgebras.

(ii) $\Rightarrow$ (iii) It follows by Proposition \ref{min are 2trans}.

(iii) $\Rightarrow$ (iv) Since $|Q|>2$ and it is simple, $Q$ is not a projection quandle. According to Lemma \ref{Lem: 2trans are Latin} $Q$ is latin.

(iv) $\Rightarrow$ (i) According to Theorem 1.4 of \cite{Stein2}, $\dis(Q)$ is solvable. So $Q$ is solvable (\cite[Theorem 1.2]{CP}) and  simple, then abelian. 
\end{proof}	
Doubly-transitive quandles described in \cite{V} are the doubly-homogeneous quandles  for which the left multiplication group coincides with the automorphism group (they are isomorphic to $\aff(\mathbb{F}_q,\lambda)$ where $\lambda$ is a generator of the multiplicative group of $\mathbb{F}_q$). 

As a consequence of the Sylow and Lagrange properties for principal latin quandles and the characterization of simple latin quandles (Theorem \ref{minimal iff}) we have the following corollary.
\begin{corollary}\label{on min subq}
Let $Q$ be a finite principal latin quandle and let $p$ be a prime dividing $|Q|$. Then there exists a strictly simple subquandle of $Q$ of size a power of $p$.
\end{corollary}
\begin{proof}
Let $p$ be a prime dividing $|Q|$. Then there exists a $p$-Sylow subquandle of $S_p$ and it is principal and latin (\ref{Lagrange and Sylow}). Then every minimal subquandle of $S_p$ (with respect to inclusion) is a strictly simple subquandle of size a power of $p$, since $S_p$ has the Lagrange property.
\end{proof}


%
\subsection{Classification of Doubly-homogeneous quandles}
In this section we show that all finite doubly homogeneous quandles are powers of a strictly simple quandles. 
The automorphism group of a doubly transitive quandle acts transitively on the set of $2$-generated subquandles, so they are all isomorphic and this property holds also in all factors. Recall that finite doubly homogeneous quandles are either projection or latin (see Lemma \ref{Lem: 2trans are Latin}) and that latin quandles are solvable. In the following we denote by $Sg(X)$ the subquandle generated by a subset $X$.
\begin{lemma}\label{2-gen sub}
Let $Q$ be a finite doubly homogeneous quandle. Then there exists a strictly simple quandle $M$ such that every $2$-generated subquandle of $Q/\alpha$ is isomorphic to $M$ for every $\alpha\in Con(Q)$.
\end{lemma}
\begin{proof}
If $Q$ is projection, then $Sg([a]_\alpha,[b]_\alpha)\cong \mathcal{P}_2$ for every $\alpha\in Con(Q)$ and every $[a]_\alpha,[b]_\alpha \in Q/\alpha$.

Let $Q$ be latin and let $M$ be a strictly simple quandle containing $a,b\in Q$, then $M=Sg(a,b)$ and $h(Sg(a,b))=Sg(h(a),h(b))\cong M$ for every $h\in\aut{Q}$. Since $\aut{Q}$ is doubly transitive, every pair of elements of $Q$ generates a subquandle isomorphic to $M$. Let $\alpha\in Con(Q)$, then $Sg([a]_\alpha,[b]_\alpha)$ is the image of $M$ with respect to the canonical map $a\mapsto [a]_\alpha$. The quandle $M$ is simple, so whenever $[a]_\alpha\neq [b]_\alpha$ then $Sg([a]_\alpha,[b]_\alpha)\cong M$.
\end{proof}

Let $Q$ be a finite latin doubly homogeneous quandle. We define the {\it minimal quandle of $Q$} as the unique strictly simple quandle in $\textbf{H}\textbf{S}(Q)$ up to isomorphism. We denote it by $M_Q$. 

\begin{proposition}\label{doubly-trans are nilp}
Let $Q$ be a finite latin doubly homogeneous quandle. Then $Q$ is a nilpotent quandle of prime power size.
\end{proposition}

\begin{proof}
Assume that $|M_Q|$ is a power of $p$. We prove by induction on the solvability length of $Q$ that if all $2$-generated subquandles are strictly simple and isomorphic then $|Q|$ is a power of $p$. If $Q$ is abelian, then in particular it is principal and latin. If $p,q$ are primes dividing $|Q|$ then $|Q|$ has a strictly simple subquandle of size a power of $p$ and of size a power of $q$ by virtue of Corollary \ref{on min subq}. Then $p=q$ since all $2$ generated subquandles are isomorphic and so the size of $Q$ is a power of $p$. Let $Q$ be solvable of length $n$, i.e. $\gamma^{n-1}(Q)$ is abelian. The blocks of $\gamma^{n-1}(Q)$ are abelian latin subquandles and then affine. Using again Corollary \ref{on min subq} the blocks have size a power of $p$. According to Lemma \ref{2-gen sub} every two-generated subquandle in the factor $Q/\gamma^{n-1}(Q)$ is isomorphic to $M_Q$ and then by induction $Q/\gamma^{n-1}(Q)$ has size a power of $p$. So $|Q|=|Q/\gamma^{n-1}(Q)||[a]|$ is a power of $p$. Finally, by \cite[Theorem 1.4]{CP}, $Q$ is nilpotent.
\end{proof}

%
%
 Let $M=\aff(\mathbb{Z}_p^m,f)$ be a strictly simple quandle.  Then its powers are given by
 $$M^n\cong \aff(\underbrace{\mathbb{Z}_p^{m}\times \ldots  \mathbb{Z}_p^m}_{n},\underbrace{f\times\ldots\times	f}_n).$$ 
The subquandles of $M^n$ are subspaces invariant under $f^{\times n}=f\times \ldots\times f$ and so they correspond to subrepresentations of the cyclic group generated by $f^{\times n}$. We can apply Maschke theorem for group representations to $M^n$, since the order of $f^{\times n}$ is coprime with $p$ and so its subrepresentation are direct product of irreducible representations, which are all isomorphic to $M$. Therefore every subquandle of $M^n$ is isomorphic to $M^k$ for some $k\leq n$.\\
On the other hand $M$ can be understood as $M=\aff(\mathbb{F}_{p^m},\lambda)$ and $M^n$ as $\aff(\mathbb{F}_{p^m}^n,\lambda I)$, where $\lambda I$ denote the scalar multiplication by $\lambda$. 
%
\begin{remark}\label{SI}
In the next theorem we are using a well-know concept in universal algebra: every algebra embeds into a direct product of some of its subdirectly irreducible (SI) factors (an algebra is called subdirectly irreducible if the intersection of all the proper congruences of $A$ is non-trivial). \\
In particular if $Q=\aff(A,f)$ and $M,N$ are strictly simple subquandles of $Q$ containing $0$, then $M\bigcap N=\{0\}$ and so $\sim_M\wedge \sim_N=0_Q$. So, if $Q$ is finite and SI then there exists a unique strictly simple subquandle containing $0$.
\end{remark}

\begin{theorem}\label{doubly-hom are power of min}
Let $Q$ be a finite latin quandle. The following are equivalent:
\begin{itemize}
\item[(i)] $Q$ is a doubly homogeneous quandle. 
\item[(ii)] $Q\cong M_Q^k$ where $M_Q$ is the minimal quandle of $Q$.
\end{itemize}
\end{theorem}
\begin{proof}
(i) $\Rightarrow$ (ii) First we show that $Q$ is principal. According to Proposition \ref{doubly-trans are nilp}, $Q$ is nilpotent and so $0_Q\neq \zeta_Q\leq \sigma_Q$ and the classes of $\sigma_Q$ are blocks with respect to the action of $\aut{Q}$. The automorphism group of $Q$ is doubly-transitive and then primitive, so $\sigma_Q=1_Q$, i.e. $Q$ is principal and latin. \\
Let $\Q(G,f)$ be a principal representation of $Q$. Since $\aut{Q}_1=C_{\aut{G}}(f)$ is transitive on $G\setminus \{1\}$ then $G$ is elementary abelian (it is nilpotent and it has no proper characteristic subgroups) and so $Q$ and all its factors are affine.\\
Let $Q/\alpha$ be a subdirectly irreducible factor of $Q$. According to Lemma \ref{2-gen sub} every pair of elements of $Q/\alpha$ generates a strictly simple subquandle isomorphic to $M_Q$. By Remark \ref{SI} every pair of elements $0,a\in Q/\alpha$ generates the same strictly simple subquandles. Therefore $Q/\alpha\cong	M_Q$. \\
The quandle $Q$ embeds into a product of some of its subdirectly irreducible factors, thus $Q$ embeds into $M_Q^n$ for some $n\in \mathbb{N}$ and therefore $Q\cong M_Q^k$ for some $k\leq n$.

(ii) $\Rightarrow$ (i) Let $q=p^n$. Up to isomorphism, there exists $\lambda\in \mathbb{F}_{q}^*$ such that $M_Q\cong \aff(\mathbb{F}_{q},\lambda)$ and $Q\cong \aff(\mathbb{F}_q^n,\lambda I)$. Then $\GL{n}{\mathbb{F}_{q}}\leq C_{\GL{nm}{p}}(\lambda I)=\aut{Q}_0$. Therefore, $\aut{Q}_0$ is transitive over $Q\setminus \{0\}$.
\end{proof}
%
%
%
%
\section{Cyclic Quandles}
\subsection{Extensions of strictly simple quandles by projection quandles} \label{sec_extensions}

 
Let $Q$ be a connected quandle and $\alpha\in Con(Q)$. We say that $Q$ a {\it connected extension} of $Q/\alpha$ by $[a]_\alpha$. We study a particular class of extensions in the same direction of \cite{Claw}, in which extensions of affine quandles by projection quandles of size $2$ has been investigated. Our focus is on extensions of strictly simple quandles by projection quandles of prime size.
%
%
%
We first shows the properties of the congruence lattice of such extensions using the Galois connection between the congruence lattice of a quandle and the lattice of normal subgroups of the left multiplication group contained in the displacement group (denoted by $Norm(Q)$). Recall that for faithful quandles a congruence $\alpha$ is abelian (resp. central) if and only if $\dis_\alpha$ is abelian (resp. central), \cite[Corollary 5.4]{CP}. Moreover if $\alpha$ is a minimal congruence then $\dis_{\alpha}$ is a minimal element of $Norm(Q)$: indeed if $N< \dis_\alpha$ then $\mathcal{O}_N<\alpha$ and so $\mathcal{O}_N=0_Q$ and accordingly $N=1$. \\
Recall that $Q$ is connected and $Q/\alpha$ is generated by $\setof{[a_i]}{1\leq i\leq n}$ then $Q$ is generated by $a_1$ together with the blocks of $a_2,\ldots,a_n$ \cite[Lemma 6.1]{GB}.
\begin{lemma}\label{abelianness of proj blocks and ss factor}
	Let $Q$ be a finite faithful connected extension of a strictly simple quandle $Q/\alpha$ by a projection quandle of prime power size. If $\alpha$ is a minimal congruence then $\alpha$ is abelian. \end{lemma}
\begin{proof}
	To prove that $\alpha$ is abelian we just need to show that $\dis_\alpha$ is abelian, since $Q$ is faithful. The subgroup $\dis_\alpha$ is generated by the family of non-trivial subgroups $K_{[a]}=\langle \setof{L_b L_c^{-1}}{b,c\in [a]}\rangle$ for $[a]\in Q/\alpha$, which are all normal in $\dis^\alpha$. The length of the orbits of $K_{[a]}$ divides the length of the orbits of $\dis^\alpha$ which is $|[a]|$. The subgroup $K_{[a]}$ acts trivially on $[a]$ and since $Q=Sg([a],b)$ whenever $b\notin [a]$, then it acts semiregularly on the block $[b]$ and then $|K_{[a]}|$ divides the size of the blocks and so it is a $p$-group. Since $K_{[a]}\bigcap K_{[b]}=1$, hence $[K_{[a]},K_{[b]}]=1$. So $\dis_\alpha$ is generated by a family of commuting $p$-subgroups and so it is nilpotent. The derived subgroup of $\dis_\alpha$ is in $ Norm(Q)$ and it is a proper subgroup of $\dis_\alpha$. The congruence $\alpha$ is minimal, therefore the derived subgroup of $\dis_\alpha$ is trivial.
\end{proof}

We will denote the first congruence $\gamma_1(Q)$ of the lower central series of $Q$ simply by $\gamma_Q$.
\begin{lemma}\label{then 3 chain}
Let $Q$ be a finite connected extension of a strictly simple quandle $Q/\alpha$ by a prime size projection quandle. Then $\alpha=\gamma_Q$ and it is the unique proper congruence of $Q$. If $Q$ is faithful $\zeta_Q=0_Q$.
\end{lemma}
\begin{proof}
	The congruence $\alpha$ is a minimal congruence of $Q$ since its blocks have prime size and it is maximal since the factor $Q/\alpha$ is simple. The quandle $Q$ is not affine since it contains projection subquandles. The factor $Q/\alpha$ is abelian so $\gamma_Q\leq \alpha$ and then $\alpha=\gamma_Q$. Let $\beta\neq \alpha$ be a congruence of $Q$. Then $\alpha\wedge \beta=0_Q$ and $Q$ embeds into $Q/\alpha\times Q/\beta$. The factor $Q/\alpha$ has no proper subquandle, so $[a]_\beta$ projects onto $Q/\alpha$. Therefore $|Q|=|Q/\beta||[a]_\beta|\geq |Q/\beta||Q/\alpha|$ and accordingly $Q\cong Q/\alpha\times Q/\beta$ and $Q/\beta\cong [a]_\alpha$ is not connected, contradiction. Hence $\alpha$ is the unique proper congruence of $Q$. If $Q$ is faithful, the blocks of central congruences are connected according to \cite[Corollary 3.2]{GB}. Then $\alpha$ is not central and so $\zeta_Q=0_Q$.
\end{proof}

The properties of the congruence lattice influence the structure of the displacement group and its subgroups.

\begin{corollary}\label{coprime size} Let $Q$ be a finite faithful connected extension of the strictly simple quandle $Q/\gamma_Q$ by a projection quandle of prime size $p$. Then
$|Q/\gamma_Q|$ and $p$ are coprime and $Z(\dis(Q))=1$.
\end{corollary}

\begin{proof}
Connected quandles of prime power size are nilpotent (\cite[Proposition 6.5]{CP}) and the orbits of $Z(\dis(Q))$ are contained in the block of $\zeta_Q$ (\cite[Lemma 5.12]{CP}). According to Lemma \ref{then 3 chain} $\zeta_Q=0_Q$ and then $Z(\dis(Q))=1$. So $|Q|$ is not a power of a prime and then $p$ is coprime with $|Q/\gamma_Q|$. 
\end{proof}

\begin{lemma}\label{cyclic then central}
	Let $Q$ be a faithful connected quandle and $\alpha\in Con(Q)$. If $\dis_\alpha$ is cyclic, then $\alpha$ is central.
\end{lemma}
\begin{proof}
	Let $\psi:\lmlt(Q)\longrightarrow \aut{\dis_\alpha}$, be the automorphism that defines the conjugation action of $\lmlt(Q)$ on $\dis_\alpha$. Since $\aut{\dis_\alpha}$ is abelian, then $\dis(Q)= \gamma_1(\lmlt(Q))\leq ker(\psi)$. So $\dis_\alpha$ is central in $\dis(Q)$ and $\alpha$ is central.
\end{proof}

Recall that if $\alpha$ is abelian, then $\left(\dis_\alpha\right)_a=\left(\dis_\alpha\right)_b$ whenever $b\, \alpha\, a$ \cite[Theorem 1.2]{CP} and in particular $\left(\dis_\alpha\right)_a$ is normal in $\dis^\alpha$.

\begin{proposition}\label{proj sub then Ab}
Let $Q$ be a finite faithful connected extension of the strictly simple quandle $Q/\gamma_Q$ by a projection quandle of prime size $p$. Then $\dis_{\gamma_Q}\cong \mathbb{Z}_p^2$ and $N(\dis(Q)_a)\leq N(\left(\dis_\alpha\right)_a)= \dis^{\gamma_Q}$ for every $a\in Q$.
%
\end{proposition}
\begin{proof}
Let $\alpha=\gamma_Q$. The quandle $Q$ is generated by $[a],b$ whenever $b\notin [a]$, and since $\left( \dis_\alpha\right)_a$ fixes the block $[a]$ then $\left( \dis_\alpha\right)_{a,b}=1$ whenever $b\notin [a]$.	Moreover $\left(\dis_\alpha\right)_a$ is normal in $\dis^{\alpha}$ since $\alpha$ is abelian.\\
The subgroup $\dis_\alpha$ is not trivial, so $\mathcal{O}_{\dis_\alpha}= \alpha$, i.e. $\dis_\alpha$ is transitive on each block of $\alpha$ and so $|\dis_\alpha|=p|\left( \dis_\alpha\right)_a|$. Therefore $|b^{\left( \dis_\alpha\right)_a}|=[\left( \dis_\alpha\right)_a:\left( \dis_\alpha\right)_{a,b}]=|\left( \dis_\alpha\right)_a|$ and since the stabilizer $\left(\dis_\alpha\right)_a$ is normal in $\dis_\alpha$ then $|\left(\dis_\alpha\right)_a|$ divides $p$. If $\dis_\alpha$ is cyclic then $\alpha$ is central by Lemma \ref{cyclic then central}. According to Lemma \ref{then 3 chain} the center of $Q$ is trivial and so $\dis_\alpha\cong \mathbb{Z}_p^2$. \\
 If $h\in N(\dis(Q)_a)$ then $h\in N(\left( \dis_\alpha\right)_a)$ since $\dis_\alpha$ is normal. If $h$ normalizes $\left(\dis_\alpha\right)_a$, i.e. $\left(\dis_\alpha\right)_{h(a)}=\left(\dis_\alpha\right)_{a}$ then $h\in \dis^\alpha$ since $\left(\dis_\alpha\right)_{a,b}=1$ whenever $b\notin [a]$. So $N(\dis(Q)_a)\leq N(\left( \dis_\alpha\right)_a)=\dis^\alpha$.
%
\end{proof} 
We are going to discuss the properties of the extensions in this section according to the equivalence $\sigma_Q$. 

\begin{lemma}\label{dis_a,b=1 equivalent}
Let $Q$ be a finite faithful connected extension of a strictly simple quandle $Q/\gamma_Q$ by a projection quandle of prime size.  Then either $\sigma_Q= \gamma_Q$ or $\sigma_Q=0_Q$. In particular, the following are equivalent:
\begin{itemize}
\item[(i)] $\dis_{\gamma_Q}=\dis^{\gamma_Q}$. 
\item[(ii)] ${\gamma_Q}=\sigma_Q$.
\item[(iii)] $\dis(Q)_{a,b}=1$ whenever $b\notin [a]_\alpha$. 
\end{itemize}
\end{lemma}

\begin{proof}
	Let $\alpha=\gamma_Q$. The block $[a]_{\sigma_Q}$ is the orbit of $a$ under the normalizer of $\dis(Q)_a$ which is contained in $\dis^\alpha$ by Proposition \ref{proj sub then Ab}(ii). Therefore $\sigma_Q \subseteq \alpha$.
	The size of the $\sigma_Q$ blocks divides the size of $Q$ then either $\sigma_Q=\alpha$ or $\sigma_Q=0_Q$. 

(i) $\Rightarrow$ (ii) Let $b\in [a]$. Then $\dis(Q)_a=\left( \dis_\alpha\right)_a=\left( \dis_\alpha\right)_b=\dis(Q)_b$, so $\alpha\leq\sigma_Q$. 

(ii) $\Rightarrow$ (iii) If $h\in \dis(Q)_{a,b}$, then $h$ fixes $[a]_\alpha$ and $b$ and accordingly $h=1$ ($Q$ is generated by $[a]_\alpha$ and $b$).

(iii) $\Rightarrow$ (i) The subgroup $|\dis(Q)_a|=|b^{\dis(Q)_a}|\leq p$. Therefore $|\dis^\alpha|=p|\dis(Q)_a|\leq p^2$. So $\dis^\alpha=\dis_\alpha$ since $|\dis_\alpha|=p^2$ (Proposition \ref{proj sub then Ab}(i)). 
\end{proof}

If $Q$ is a faithful connected extension of a strictly simple quandle $Q/\gamma_Q$ by a projection quandle of prime size then $\dis^{\gamma_Q}=\gamma_1(\dis(Q))$ is a maximal characteristic subgroup  and $Q/\gamma_Q$ is affine over $G/\gamma_1(G)$ by \cite[Proposition 2.6]{GB}. The subgroup $\dis_{\gamma_Q}$ is a minimal normal subgroup of $\lmlt(Q)$ and since $\gamma_2	(G)\leq \dis_{\gamma_Q}$ \cite[Proposition 3.3]{CP} then $\gamma_2(G)= \dis_\alpha$ (because $G$ is not nilpotent). Moreover $\gamma_2(\dis(Q))$ and $\dis(Q)/\gamma_1(\dis(Q))$ are elementary abelian group with coprime size (see Corollary \ref{coprime size} and Proposition \ref{proj sub then Ab}(i)).\\
The next Proposition describes the structure of a group with such properties.

\begin{proposition}\label{lemma on special q groups}
Let $p,q$ be different primes, $m,n\in \mathbb{N}$ and let $G$ be a group such that $\gamma_2(G)\cong \mathbb{Z}_p^m$, $G/\gamma_1(G)\cong \mathbb{Z}_q^n$ and $\gamma_1(G)$ is a maximal characteristic subgroup. Then the exact sequence
$$1\longrightarrow \gamma_2(G)\longrightarrow G\longrightarrow G/\gamma_2(G)=K$$
splits and $K$ is a special $q$-group with $Z(K)\cong \gamma_1(G)/\gamma_2(G)$. In particular if $\gamma_1(G)=\gamma_2(G)$ then $G\cong \mathbb{Z}_p^m\rtimes_{\rho} \mathbb{Z}_q^m$.
\end{proposition} 

\begin{proof}
If $\gamma_1(G)=\gamma_2(G)$ then the short exact sequence 
\begin{displaymath}
1\longrightarrow\gamma_1(G) \longrightarrow G \longrightarrow G/\gamma_1(G)\longrightarrow 1
\end{displaymath}
splits as a semidirect product, since $\gamma_1(G)$ is abelian and the order of $G/\gamma_1(G)$ and $\gamma_1(G)$ are coprime. 

Otherwise, the factor group $K=G/\gamma_2(G)$ is a 2-step nilpotent group and so $\gamma_1(K)\leq Z(K)$. Let $\pi:G\longrightarrow G/\gamma_2(G)$  be the canonical projection, then the characteristic subgroup $L=\pi^{-1}(Z(K))$ contains $\gamma_1(G)$. The subgroup $L$ is a proper subgroup of $G$ since $K$ is not abelian. 
Then $L=\gamma_1(G)$ and $Z(K)=\gamma_1(G)/\gamma_2(G)$. 

Since $K$ is nilpotent of length $2$, the commutator mappings 
$$[-,a]:K\longrightarrow Z(K), \quad x\mapsto [x,a]$$ 
are homomorphisms and they factorize through the canonical projection onto $K/Z(K)\cong G/\gamma_1(G)$ which is an elementary $q$-abelian group. Moreover if $\setof{a_i}{1\leq i \leq n}$ is a basis of $G/\gamma_1(G)$ then $Z(K)\cong \gamma_1(G)/\gamma_2(G)$ is generated by $\setof{[a_i,a_j]\gamma_2(G)}{1\leq i,j\leq n}$ \cite[Proposition 9.2.5]{Sims}. Since the generators are images of the factorized commutator mappings they have order $q$. Then $Z(K)$ is an elementary $q$-abelian group and $K$ is a $q$ group. The factor with respect to $\gamma_1(K)$ is elementary abelian then $\gamma_1(K)=\Phi(K)$, so $K$ is a special $q$-group. Thus, the short exact sequence 
\begin{displaymath}
1\longrightarrow\gamma_2(G) \longrightarrow G \longrightarrow G/\gamma_2(G)\longrightarrow 1
\end{displaymath}
splits as a semidirect product.  
%
\end{proof}

We can apply Proposition \ref{lemma on special q groups} to the displacement group of extensions of strictly simple quandles by a prime size projection quandle.  Note that the two cases in Proposition \ref{lemma on special q groups} correspond exactly to the two cases in Lemma \ref{dis_a,b=1 equivalent}. By virtue of Proposition \ref{decomposition}, $\sigma_Q=\gamma_Q$ if and only if $N(\dis(Q)_a)=\dis^{\gamma_Q}$ and $\sigma_Q=0_Q$ if and only if $N(\dis(Q)_a)=\dis(Q)_a$. The actions in Proposition \ref{lemma on special q groups} can be extended to the whole group $G$, just by setting $\rho_x=1$ for every $x\in \gamma_2(G)$.

\begin{theorem}\label{extension of minimal by proj of prime}
Let $Q$ be a finite faithful connected extension of a strictly simple quandle $Q/\gamma_Q$ of size $q^n$ by a projection quandle of prime size $p$. Then:
\begin{itemize}
\item[(i)] if $\sigma_Q=\gamma_Q$ then $|Q|=6$.

\item[(ii)] If $\sigma_Q=0_Q$ then $q=2$, $n=2k$, $p=2^k+1$ and 
\begin{equation*}\label{structure of DIS_2}
\dis(Q)\cong \mathbb{Z}_{2^k+1}^2\rtimes_{\rho} H
\end{equation*}
where $H$ is an extraspecial $2$-group and the action $\rho$ is faithful.
\end{itemize}
\end{theorem}

\begin{proof}
Let $\alpha=\gamma_Q$, $G=\dis(Q)$, $f=\widehat{L}_a\in \aut{G}$ and $H=Fix(f)=\dis(Q)_a$. 
The block $[a]_\alpha$ is projection and it is isomorphic to $\mathcal{Q}(\dis_\alpha, \left(\dis_\alpha\right)_a, f|_{\dis_\alpha})$. Fixed a suitable basis $e_1, e_2$ with $e_1\in \left(\dis_\alpha\right)_a$, the restriction of $f$ to $\dis_\alpha$ is
\begin{equation}\label{restriction of f}
\begin{small}
F=f|_{\dis_\alpha}=\begin{bmatrix}
1 & 1\\
0 & 1
\end{bmatrix}.
\end{small}
\end{equation} We can apply Proposition \ref{lemma on special q groups} to $G$ and we have two cases. 

(i) If $\sigma_Q=\gamma_Q$ then $\gamma_1(G)=\gamma_2(G)=N(Fix(f))$. So $G\cong \mathbb{Z}_p^m\rtimes_{\rho} \mathbb{Z}_q^n$ and $\rho$ is faithful since $Z(G)=1$. 
If $h\in  \mathbb{Z}_q^n$ then
\begin{eqnarray}\label{action of h}
\begin{small}
\rho_h=\begin{bmatrix}
x & w\\
y & t
\end{bmatrix}
\end{small}\end{eqnarray}
with $y\neq 0$ since $h$ does not normalize $Fix(f)$. Using that $\rho_{f(h)}=f\rho_h f^{-1}$ and that $\rho_h \rho_{f(h)}-\rho_{f(h)}\rho_h=0$, we have 
\begin{eqnarray*}
\left(F\rho_h F^{-1} \rho_h-\rho_h F\rho_h F^{-1}\right)_{2,1}&=&
 -2y^2=	0.
\end{eqnarray*}
Then $2y^2=0$ and so $p=2$. Hence $q^n$ divides $|GL_2(2)|=6$. Therefore $q^n=3$ and $|Q|=6$.

(ii) If $\sigma_Q=\gamma_Q$ then $\gamma_1(G)\neq \gamma_2(G)$ and $N(Fix(f))=Fix(f)$. Then $G\cong \gamma_2(G)\rtimes_\rho K\cong  \mathbb{Z}_p^2\rtimes_{\rho} K$ and $K$ is a special $q$-group. If $\rho_h=1$ then $$h\in C_G(\gamma_1(G))\bigcap K\leq N(Fix(f))\bigcap K  \leq \gamma_1(G)\bigcap K= Z(K),$$ where the last inclusion follows by Proposition \ref{proj sub then Ab}. So $h\in Z(G)=1$ and the action $\rho$ is faithful. \\
 Using the equation of the orbits for the action of $z\in Z(K)$ over the $1$-dimensional subspaces of $\gamma_2(G)$ we have that 
\begin{equation}\label{equation of classes}
p+1=e+nq
\end{equation}
where $e$ is the number of the eigenspaces of $\rho_z$. Since $z$ normalizes $\left(\dis_\alpha\right)_a$ then $e=2$. If the eigenvalues of $\rho_z$ are different, then $K$ embeds in the centralizer of $\rho_z$ which is the abelian subgroup of the diagonal matrices. Since $K$ is not abelian then $\rho_z$ is a scalar matrix. 
%
%
%
Therefore $Z(K)$ embeds into the cyclic subgroup $\mathbb{Z}_{p-1}$ and accordingly it is cyclic. Then $K$ is an extraspecial $q$-group, $q$ divides $p-1$ and $n=2k$. \\
Since $q$ divides $p-1$, by virtue of \eqref{equation of classes} the action of a non central element $h\in K$ is diagonalizable with different eigenvalues $a$ and $b$. Let $z$ be a generator of $Z(K)$ and $g\in K$ which does not centralize $h$, then $[\rho_g, \rho_h]=\rho_z^s$ for some $1\leq s\leq p-1$, i.e. there exists $1\leq k\leq p-1$ such that
%
\begin{eqnarray*}
\begin{small}
\begin{bmatrix}
x & w\\
y & t
\end{bmatrix}^{-1}
\begin{bmatrix}
a & 0\\
0 & b
\end{bmatrix}^{-1}
\begin{bmatrix}
x & w\\
y & t
\end{bmatrix}
\begin{bmatrix}
a & 0\\
0 & b
\end{bmatrix}=
\frac{1}{ab\det(\rho_g)}\begin{bmatrix}
a(btx-awy)	& b(b-a)tw\\
a(a-b)xy	 & b(atx-bwy)
\end{bmatrix} =\begin{bmatrix}
 k & 0\\
 0 & k
\end{bmatrix}.
\end{small}\end{eqnarray*}
The elements $h$ and $g$ do not commute, hence $x=t=0$ and $a=-b$. This implies that $k=-1$ and therefore $q=2$. 
 	Let $s$ be the order of left multiplications of $Q/\gamma_Q$. Then $f^s(h)=h z^r d$ for some $r=0,1$ and $d\in \gamma_2(G)$ for every non central element $h\in H$. Then
 \begin{equation}\label{condition on ro_x}
 \rho_{f^s(h)}=f^{s}\rho_h f^{-s}=\rho_h\rho_z^r,
 \end{equation}
where $\rho_h$ is like in \eqref{action of h} and $y\neq 0$ since $h$ do not normalize $\left(\dis_{\gamma_Q}\right)_{a}$ (Proposition \ref{proj sub then Ab}). Therefore we have
\begin{eqnarray*}
\left(\rho_h ^{-1}f^s \rho_h f^{-s}\right)_{2,1}=
\begin{small}
\left(\begin{bmatrix}
x & w\\
y & t
\end{bmatrix}^{-1}
\begin{bmatrix}
1 & s\\
0 & 1
\end{bmatrix}\begin{bmatrix}
x & w\\
y & t
\end{bmatrix}\begin{bmatrix}
1 & -s\\
0 & 1
\end{bmatrix}\right)_{2,1}=\frac{sy^2}{\det(\rho_h)}=\left(\rho_z^r\right)_{2,1}=0.
\end{small}
\end{eqnarray*}
Hence $s=0 \pmod p$, therefore $p$ divides $s$. The factor $Q/\gamma_Q$ is strictly simple, then $s$ divides $|Q/\alpha|=2^{2k}-1=(2^k-1)(2^k+1)$ and so 
 \begin{center}
 	 $p$ divides either $2^{k}+ 1$ or $2^{k}- 1$. (*)
 \end{center}
 The centralizer of $\rho_h$ have size $2^{2k}$ for every non central element $h\in K$. According to \eqref{equation of classes} the number of eigenspaces of $h$ is either $0$ or $2$. If the action of $\rho_h$ is irreducible, the centralizer of $\rho_h$ in the image of $\rho$ is cyclic \cite[Theorem 2.3.5]{256} and then it has order $4$ and $k=1$. Since $p$ divides $4-1=3$ then $p=3=2^1+1$.\\
 If $\rho_h$ is diagonalizable with different eigenvalues ($h$ is not central) then the centralizer of $\rho_h$ is given by the diagonal matrices so $2^{2k}$ divides $(p-1)^2$. Hence $2^k$ divides $p-1$, and in particular $p \geq 2^k+1$. From (*), it follows $p=2^k+1$.
%
%
%
%
\end{proof}

In the RIG library \cite{RIG} the unique example of quandles as in Theorem \ref{extension of minimal by proj of prime}(ii) is SmallQuandle(12,10).

%

\subsection{Classification of Connected Cyclic Quandles}\label{Sec:cyclic}
In this section we investigate a class of quandles defined by a particular cyclic structure of the left multiplications. This class has been already studied in \cite{CQ} using combinatorial tools: we present an alternative approach making use of the results of the previous section. 
\begin{definition}\label{cyclic} A finite quandle $Q$ is called {\it cyclic with $f$ fixed points} if for every $a\in Q$, the cycle structure of $L_a$ is given by $f$ fixed points and one cycle of length $|Q|-f$. 
\end{definition}
We complete the classification of cyclic quandles in two steps: first we show that if the number of fixed points is prime then $|Q|=6$. Then we show that any cyclic quandle with $f$ fixed points has a cyclic factor with a prime number of fixed points and from that we deduce that $|Q|=6$ as well. \\
Let us recall some results from \cite{CQ}.
\begin{lemma}\label{know about cyclic}\cite[Proposition 2.4, Theorem 2.1, Theorem 2.2]{CQ}
Let $Q$ be a cyclic quandle with $f$ fixed points. Then $Q$ is connected if and only if $|Q| > 2f$ and $\P$ is an equivalence relation with $[a]_\P=Fix(L_a)$. 
\end{lemma}
In the following we are making use of the following subgroup
\begin{equation}\label{P_cyclic}
 \langle L_a\rangle^\P = \langle L_a\rangle\bigcap \mathrm{Aut}^\P.
\end{equation}

\begin{proposition}\label{doubly transitive factor}
Let $Q$ be a connected cyclic quandle with $f$ fixed points. Then $\P$ is the unique maximal congruence of $Q$ and $Q/\P$ is a doubly transitive quandle. 
\end{proposition}
\begin{proof}
Since $\langle L_a \rangle$ is regular on $Q\setminus [a]_{\P}$ then the group $\langle L_a\rangle/\LP{a}$ acts regularly on the set of $\P$-blocks $Q/\P\setminus \{[a]\}$. So we have that 
\begin{equation}\label{orders}
|Q/\P|-1=\frac{|L_a|}{|\LP{a}|}=\frac{(|Q/\P|-1)f}{|\LP{a}|}.
\end{equation}
Therefore the subgroup subgroup $\LP{a}$ has order $f$ and so it is regular on each block of $\P$ different from $[a]$. The subgroup $\lmlt^\P$ contains $\LP{a}$ for each $a\in Q$, therefore $\lmlt^P$ is transitive on each block. Then $\P=\mathcal{O}_{\lmlt^\P}$ and according to \cite[Lemma 2.6]{CP} $\P$ is a congruence. Let $a,b\in Q$ in different classes with respect to $\P$. Then $a^{L_b}\bigcup b^{L_a}\subseteq Sg(a,b)$. By virtue of Lemma \ref{know about cyclic}, $a^{L_b}=Q\setminus [b]_\P$ and $b^{L_a}=Q\setminus [a]_\P$, therefore $Q$ is generated by $a$ and $b$. Hence, the blocks of $\P$ are the unique maximal subquandles of $Q$, and accordingly $\P$ is the unique maximal congruence of $Q$.\\
 The order of left multiplication in the factor $Q/\P$ is $|Q/\P|-1$ according to \eqref{orders}. Then the factor $Q/\P$ is doubly transitive \cite[Corollary 4]{V}.
\end{proof}
%
%
%

\begin{lemma}\label{cyclic are faith}
Let $Q$ be a connected cyclic quandle. Then $Q$ is faithful.
\end{lemma}
\begin{proof}
 Assume that $L_a=L_b$. Then necessarily $a\, \P \, b$ and there exists $c\notin [a]_\P$, and $h\in \LP{c}$ such that $b=h(a)$. So $L_b=L_{h(a)}=h L_a h^{-1}=L_a$. Then $h(a\ast c)=a\ast h(c)=a\ast c$ and $[a\ast c]\neq [c]$ ($Q/\P$ is affine and connected, so it has no projection subquandle). Since $\LP{c}$ acts regularly on each block different of $[c]_\P$, then $h=1$, i.e. $a=b$.
\end{proof}
We can apply Theorem \ref{extension of minimal by proj of prime}(i) to cyclic quandles with a prime number of fixed points. Indeed if $Q$ is such a quandle, $Q$ is faithful, $Q/\P$ is strictly simple and $\dis(Q)_{a,b}=1$ whenever $b\notin [a]$ ($Q$ is generated by $a$ and $b$).
%
\begin{corollary}\label{prime and cyclic}
Let $Q$ be a connected cyclic quandle with $p$ fixed points where $p$ is a prime. Then $p=2$ and $|Q|=6$.
\end{corollary}

\begin{theorem}\label{Cyclic}
Let $Q$ be a connected cyclic quandle with $f$ fixed points. Then $f=2$ and $|Q|=6$.
\end{theorem}
\begin{proof}
The quandle $Q$ is generated by any pair of elements such that $[a]_\P\neq [b]_\P$. Then $$\LP{a}\bigcap \LP{b}\leq \lmlt(Q)_a\bigcap \lmlt(Q)_b=1$$ and so $\mathbb{Z}_f^2\cong \LP{a}\times \LP{b}\leq \lmlt^P$. Since $\lmlt(Q)_a\leq \lmlt^\P$ is regular on $[b]$, $|\lmlt^P|=f|\lmlt(Q)_a|\leq f^2$. Hence $\lmlt^P\cong \mathbb{Z}_f^2$.\\
Let $q$ be a prime dividing $f$. Then $K=q \lmlt^\P\cong q Z_{f}\times q Z_f$ is a characteristic subgroup of $\lmlt^\P$ and so it is normal in $\lmlt(Q)$. Then the congruence $\alpha=\mathcal{O}_K$ is contained in $\P$ and it has blocks of size $\frac{f}{q}$. The factor $Q/\alpha$ is a cyclic quandle, since $Q/\alpha= Fix(L_{[a]_{\alpha}}) \bigcup \setof{L_{[a]_\alpha}^k([b]_\alpha)}{k\in \mathbb{N}}$ of size $|Q/\alpha|=q|Q/\P|$. By virtue of Corollary \ref{prime and cyclic}, $q=2$ and $|Q/\P|=3$. Hence $f=2^k$ and the subgroups $2^l \lmlt^\P$ for every $0\leq l\leq k$ provide a chain of congruences with cyclic factor of size $3\cdot 2^l$. \\
An exhaustive computer search (for instance on the RIG database on GAP) shows that there are no such quandles for $k= 2$. Hence $k=1$ and $|Q|=6$.
\end{proof}

There exists just two connected quandles of size $6$ and one of them is the unique connected cyclic quandle \cite{RIG}.

\section*{Acknowledgments}
We thank Giuliano Bianco and Daniele Toller for the fruitful discussion and the useful advices.  

\bibliographystyle{abbrv}
\bibliography{references} 

\def\cprime{$'$} \def\cprime{$'$}
\begin{thebibliography}{10}

\bibitem{AG}
N.~Andruskiewitsch and M.~Gra{\~n}a.
\newblock From racks to pointed {H}opf algebras.
\newblock {\em Adv. Math.}, 178(2):177--243, 2003.

\bibitem{Russi2}
V.~Bardakov, T.~Nasybullov, and M.~Singh.
\newblock Automorphism groups of quandles and related groups.
\newblock {\em Monatsh. Math.}, 189(1):1--21, 2019.

\bibitem{2trans}
V.~G. Bardakov, P.~Dey, and M.~Singh.
\newblock Automorphism groups of quandles arising from groups.
\newblock {\em Monatsh. Math.}, 184(4):519--530, 2017.

\bibitem{Bel}
V.~D. Belousov.
\newblock {\em {\cyr Osnovy teorii kvazigrupp i lup}}.
\newblock Izdat. ``Nauka'', Moscow, 1967.

\bibitem{UA}
C.~Bergman.
\newblock {\em Universal algebra}, volume 301 of {\em Pure and Applied
  Mathematics (Boca Raton)}.
\newblock CRC Press, Boca Raton, FL, 2012.
\newblock Fundamentals and selected topics.

\bibitem{GB}
G.~{Bianco} and M.~{Bonatto}.
\newblock {On connected quandles of prime power order}.
\newblock {\em arXiv e-prints}, page arXiv:1904.12801, Apr 2019.

\bibitem{CP}
M.~{Bonatto} and D.~{Stanovsk{\'y}}.
\newblock {Commutator theory for racks and quandles}.
\newblock {\em arXiv e-prints}, page arXiv:1902.08980, Feb 2019.

\bibitem{Bruck}
R.~Bruck.
\newblock {\em A Survey of Binary Systems}.
\newblock Number v. 20 in A Survey of Binary Systems. Springer, 1958.

\bibitem{GenAlex}
W.~E. Clark, L.~A. Dunning, and M.~Saito.
\newblock Computations of quandle 2-cocycle knot invariants without explicit
  2-cocycles.
\newblock {\em J. Knot Theory Ramifications}, 26(7):1750035, 22, 2017.

\bibitem{C2}
W.~E. Clark and M.~Saito.
\newblock Algebraic properties of quandle extensions and values of cocycle knot
  invariants.
\newblock {\em J. Knot Theory Ramifications}, 25(14):1650080, 17, 2016.

\bibitem{Claw}
F.~J.~B.~J. {Clauwens}.
\newblock {Small connected quandles}.
\newblock {\em arXiv e-prints}, page arXiv:1011.2456, Nov 2010.

\bibitem{ESS}
P.~Etingof, T.~Schedler, and A.~Soloviev.
\newblock Set-theoretical solutions to the quantum {Y}ang-{B}axter equation.
\newblock {\em Duke Math. J.}, 100(2):169--209, 1999.

\bibitem{EGS}
P.~Etingof, A.~Soloviev, and R.~Guralnick.
\newblock Indecomposable set-theoretical solutions to the quantum
  {Y}ang-{B}axter equation on a set with a prime number of elements.
\newblock {\em J. Algebra}, 242(2):709--719, 2001.

\bibitem{comm}
R.~Freese and R.~McKenzie.
\newblock {\em Commutator theory for congruence modular varieties}, volume 125
  of {\em London Mathematical Society Lecture Note Series}.
\newblock Cambridge University Press, Cambridge, 1987.

\bibitem{GALK}
V.~M. Galkin.
\newblock Finite distributive quasigroups.
\newblock {\em Mat. Zametki}, 24(1):39--41, 141, 1978.

\bibitem{GALK2}
V.~M. Galkin.
\newblock Left distributive finite order quasigroups.
\newblock {\em Mat. Issled.}, (51):43--54, 163, 1979.
\newblock Quasigroups and loops.

\bibitem{Hou_Aut}
X.-d. Hou.
\newblock Automorphism groups of {A}lexander quandles.
\newblock {\em J. Algebra}, 344:373--385, 2011.

\bibitem{Hou}
X.-D. Hou.
\newblock Finite modules over {$\Bbb Z[t,t^{-1}]$}.
\newblock {\em J. Knot Theory Ramifications}, 21(8):1250079, 28, 2012.

\bibitem{HSV}
A.~Hulpke, D.~Stanovsk\'{y}, and P.~Vojt\v{e}chovsk\'{y}.
\newblock Connected quandles and transitive groups.
\newblock {\em J. Pure Appl. Algebra}, 220(2):735--758, 2016.

\bibitem{JZ2}
P.~Jedli\v{c}ka, A.~Pilitowska, D.~Stanovsk\'{y}, and A.~Zamojska-Dzienio.
\newblock Subquandles of affine quandles.
\newblock {\em J. Algebra}, 510:259--288, 2018.

\bibitem{J}
D.~Joyce.
\newblock A classifying invariant of knots, the knot quandle.
\newblock {\em J. Pure Appl. Algebra}, 23(1):37--65, 1982.

\bibitem{JS}
D.~Joyce.
\newblock Simple quandles.
\newblock {\em J. Algebra}, 79(2):307--318, 1982.

\bibitem{CQ}
A.~{Lages} and P.~{Lopes}.
\newblock {Quandles of cyclic type with several fixed points}.
\newblock {\em arXiv e-prints}, page arXiv:1803.10487, Mar 2018.

\bibitem{RIG}
Mat{\'{\i}}as Gra{\~n}a and Leandro Vendramin.
\newblock {\em {Rig, a GAP package for racks, quandles and Nichols algebras}}.

\bibitem{Matveev}
S.~V. Matveev.
\newblock Distributive groupoids in knot theory.
\newblock {\em Mat. Sb. (N.S.)}, 119(161)(1):78--88, 160, 1982.

\bibitem{McC}
J.~{McCarron}.
\newblock {Connected Quandles with Order Equal to Twice an Odd Prime}.
\newblock {\em arXiv e-prints}, page arXiv:1210.2150, Oct 2012.

\bibitem{Nel}
S.~Nelson.
\newblock Classification of finite {A}lexander quandles.
\newblock In {\em Proceedings of the {S}pring {T}opology and {D}ynamical
  {S}ystems {C}onference}, volume~27, pages 245--258, 2003.

\bibitem{256}
M.~W. Short.
\newblock {\em The primitive soluble permutation groups of degree less than
  {$256$}}, volume 1519 of {\em Lecture Notes in Mathematics}.
\newblock Springer-Verlag, Berlin, 1992.

\bibitem{Sims}
C.~C. Sims.
\newblock {\em Computation with finitely presented groups}, volume~48 of {\em
  Encyclopedia of Mathematics and its Applications}.
\newblock Cambridge University Press, Cambridge, 1994.

\bibitem{Stanos}
D.~Stanovsk\'{y}.
\newblock A guide to self-distributive quasigroups, or {L}atin quandles.
\newblock {\em Quasigroups Related Systems}, 23(1):91--128, 2015.

\bibitem{Stein2}
A.~Stein.
\newblock A conjugacy class as a transversal in a finite group.
\newblock {\em J. Algebra}, 239(1):365--390, 2001.

\bibitem{Frob2}
J.~Thompson.
\newblock Finite groups with fixed-point-free automorphisms of prime order.
\newblock {\em Proc. Nat. Acad. Sci. U.S.A.}, 45:578--581, 1959.

\bibitem{Va}
M.~A. Valeriote.
\newblock Finite simple abelian algebras are strictly simple.
\newblock {\em Proc. Amer. Math. Soc.}, 108(1):49--57, 1990.

\bibitem{V}
L.~Vendramin.
\newblock Doubly transitive groups and cyclic quandles.
\newblock {\em J. Math. Soc. Japan}, 69(3):1051--1057, 2017.

\bibitem{Vla}
J.~Vlach\'{y}.
\newblock Small left distributive quasigroups.
\newblock {\em Bachelor's Thesis}, 2010.

\end{thebibliography}
\end{document}